\documentclass[leqno,10pt]{amsart}
\usepackage{amssymb,amsmath,amsfonts,amsthm,mathtools,latexsym,ifthen,bezier}
\usepackage{graphicx}
\usepackage{amsxtra}
\usepackage{xspace}
\usepackage{url}
\usepackage[english]{babel}
\usepackage{etoolbox}
\DeclareMathOperator*{\diag}{\bigtriangleup}

\newcommand{\kla}[1]{ {\langle #1 \rangle} }

\newcommand{\st}{\;|\;}

\newcommand{\dom}{ {\rm dom} }
\newcommand{\ran}{ {\rm ran} }

\newcommand{\sub}{\subseteq}

\newfont{\ssi}{cmssi12 at 12pt}

\newcommand{\rest}{{\restriction}}
\newcommand{\cf}{ {\rm cf} }
\newcommand{\On}{ {\rm On} }

\newcommand{\verl}{{{}^\frown}}
\newcommand{\leer}{\emptyset}
\newcommand{\ohne}{\setminus}

\newcommand{\id}{ {\rm id} }

\newenvironment{ea*}{\begin{eqnarray*}}{\end{eqnarray*}}

\newcounter{claimnumber}
\AtEndEnvironment{proof}{\setcounter{claimnumber}{0}}
\newenvironment{NumberedClaim}%
{\begin{enumerate}[label=(\arabic*)]
\refstepcounter{claimnumber}
\setcounter{enumi}{\value{claimnumber}-1}
\item}%
{\end{enumerate}}

\setbox0=\hbox{$\longrightarrow$}
\newcommand{\To}{\longrightarrow}

\newcommand{\Lim}{{\rm Lim}}

\newcommand{\power}{{\mathcal{P}}}

\newcommand{\calM}{\mathcal{M}}

\newcommand{\bG}{{\bar{G}}}

\newcommand{\bP}{{\bar{\P}}}

\newcommand{\bS}{{\bar{S}}}
\newcommand{\bT}{{\bar{T}}}

\newcommand{\bgamma}{{\bar{\gamma}}}

\newcommand{\bkappa}{{\bar{\kappa}}}

\newcommand{\blambda}{{\bar{\lambda}}}

\newcommand{\barf}{{\bar{f}}}

\newcommand{\tQ}{{\tilde{Q}}}

\newcommand{\bN}{{\bar{N}}}

\newcommand{\tkappa}{{\tilde{\kappa}}}

\newcommand{\vp}{{\vec{p}}}

\newcommand{\vA}{{\vec{A}}}

\newcommand{\vR}{{\vec{R}}}
\newcommand{\vS}{{\vec{S}}}
\newcommand{\vT}{{\vec{T}}}

\newcommand{\seq}[2]{{\langle#1\;|\;}\linebreak[0]{#2\rangle}}

\renewcommand{\phi}{\varphi}
\newcommand{\card}[1]{\overline{\overline{#1}}}

\newcommand{\ZFC}{\ensuremath{\mathsf{ZFC}}\xspace}

\newcommand{\ZFCm}{\ensuremath{{\ZFC}^-}\xspace}

\newcommand{\V}{\ensuremath{\mathrm{V}}\xspace}

\newcommand{\forces}{\Vdash}

\def\<#1>{\langle#1\rangle}

\newcommand{\B}{{\mathord{\mathbb{B}}}}
\renewcommand{\P}{{\mathord{\mathbb P}}}
\newcommand{\Q}{{\mathord{\mathbb Q}}}

\newcommand{\SC}{\ensuremath{\mathsf{SC}}\xspace}

\newcommand{\MM}{\ensuremath{\mathsf{MM}}\xspace}

\newcommand{\FA}{\ensuremath{\mathsf{FA}}}
\newcommand{\SCFA}{\ensuremath{\mathsf{SCFA}}\xspace}

\newcommand{\BFA}{\ensuremath{\mathsf{BFA}}}

\newcommand{\MP}{\ensuremath{\mathsf{MP}}}

\newcommand{\ColNothing}{\mathrm{Col}}
\newcommand{\Col}[1]{\ColNothing(#1)}

\newcommand{\MPColNothing}[1]{\MP_{\Col{\dot{\kappa}}}}

\newcommand{\CH}{\ensuremath{\mathsf{CH}}\xspace}

\newcommand{\IA}{\mathsf{IA}}

\newcommand{\Refl}{\ensuremath{\mathsf{Refl}}}

\newcommand{\OSR}[1]{\ensuremath{\mathsf{OSR}_{#1}}}
\newcommand{\DSR}[1]{\ensuremath{\mathsf{DSR}{(#1)}}}
\newcommand{\SRP}{\ensuremath{\mathsf{SRP}}\xspace}

\newcommand{\BSCFA}{\ensuremath{\mathsf{BSCFA}}\xspace}

\newcommand{\Todorcevic}{Todor\v{c}evi\'{c}\xspace}

\newtheorem{thm}{Theorem}[section]
\newtheorem*{thm*}{Theorem} 
\newtheorem{cor}[thm]{Corollary}
\newtheorem{lem}[thm]{Lemma}
\newtheorem{obs}[thm]{Observation}

\newtheorem{fact}[thm]{Fact}

\theoremstyle{definition}
\newtheorem{defn}[thm]{Definition}
\newtheorem{question}[thm]{Question}

\newcommand{\thistheoremname}{}
\newtheorem{genericthm}[thm]{\thistheoremname}

\theoremstyle{remark}
\newtheorem{remark}[thm]{Remark}

\usepackage{enumitem}

\newcommand{\SSP}{\ensuremath{\mathsf{SSP}}\xspace}
\newcommand{\infSC}{\ensuremath{\mathsf{\infty\text{-}SC}}\xspace}

\newcommand{\DRP}{\ensuremath{\mathsf{DRP}}\xspace}
\newcommand{\lifting}{\mathsf{lift}}
\renewcommand{\DSR}{\ensuremath{\mathsf{DSR}}}
\newcommand{\eDSR}{\ensuremath{\mathsf{eDSR}}}
\newcommand{\uDSR}{\ensuremath{\mathsf{uDSR}}}
\newcommand{\sDSR}{\ensuremath{\mathsf{sDSR}}}
\newcommand{\projectdown}{\mathbin{\downarrow}}
\newcommand{\projectup}{\mathbin{\uparrow}}
\renewcommand{\card}[1]{|#1|}
\DeclareMathOperator{\Tr}{\mathsf{Tr}}
\DeclareMathOperator{\eTr}{\mathsf{eTr}}

\newcommand{\DSRP}{\ensuremath{\mathsf{DSRP}}\xspace}
\newcommand{\eDSRP}{\ensuremath{\mathsf{eDSRP}}\xspace}
\newcommand{\eRefl}{\ensuremath{\mathsf{eRefl}}\xspace}

\newcommand{\wDRP}{\ensuremath{\mathsf{wDRP}}}
\newcommand{\wDRPIA}{\ensuremath{\wDRP_\IA}}

\newcommand{\infSCFA}{\ensuremath{\infty\text{-}\mathsf{SCFA}}}

\newcommand{\Stationary}{\mathsf{S}}

\newcommand{\vQ}{\vec{Q}}

\begin{document}

\title{The diagonal strong reflection principle and its fragments}
\author{Sean D.~Cox}
\address{Department of Mathematics and Applied Mathematics, Virginia Commonwealth University, 1015 Floyd Avenue, Richmond, Virginia 23284, USA}
\email{scox9@vcu.edu}
\author{Gunter Fuchs}
\thanks{Acknowledgments of the second author: support for this project was provided by PSC-CUNY Award \#63516-00 51, jointly funded by The Professional Staff Congress and The City University of New York, and by Simons Award Number 580600. The second author is also grateful for the hospitality of the logic group at Virginia Commonwealth University, and in particular to Brent Cody, during a visit in August 2019.}
\address{The College of Staten Island (CUNY)\\2800 Victory Blvd.~\\Staten Island, NY 10314, USA}
\address{The Graduate Center (CUNY)\\365 5th Avenue, New York, NY 10016, USA}
\email{gunter.fuchs@csi.cuny.edu}
\urladdr{www.math.csi.cuny.edu/~fuchs}
\date{\today}
\begin{abstract}
A diagonal version of the strong reflection principle is introduced, along with fragments of this principle associated to arbitrary forcing classes. The relationships between the resulting principles and related principles, such as the corresponding forcing axioms and the corresponding fragments of the strong reflection principle are analyzed, and consequences are presented. Some of these consequences are ``exact'' versions of diagonal stationary reflection principles of sets of ordinals. We also separate some of these diagonal strong reflection principles from related axioms.
\end{abstract}
\keywords{Forcing axioms, reflection principles, strong reflection principle, diagonal reflection}
\subjclass[2020]{03E57, 03E50, 03E35, 03E75}
\maketitle

\section{Introduction}

Fuchs \cite{Fuchs:CanonicalFragmentsOfSRP} introduced fragments of \Todorcevic's strong reflection principle \SRP (see \cite[p.~57]{Bekkali:TopicsInST}) for forcing classes $\Gamma$ other than the class \SSP of all stationary set preserving forcings. The focus was on the class of all subcomplete forcings, and the goal was to find a principle that relates to the forcing axiom for $\Gamma$ in much the same way that \SRP relates to \MM, the forcing axiom for \SSP, namely such that
\begin{enumerate}[label=(\arabic*)]
	\item
	\label{item:FollowsFromFA}
	the forcing axiom for $\Gamma$, $\FA(\Gamma)$, implies $\Gamma$-\SRP,
	\item
	\label{item:GeneralizesSRP}
	letting $\SSP$ be the class of all stationary set preserving forcing notions, $\SRP$ is equivalent to $\SSP$-$\SRP$,
	\item
	\label{item:ImpliesMajorConsequences}
	letting $\SC$ be the class of all subcomplete forcing notions, $\SC$-$\SRP$ captures many of the major consequences of $\SCFA$, the forcing axiom for subcomplete forcing.
\end{enumerate}
Subcomplete forcing was introduced by Jensen \cite{Jensen:SPSCF}, \cite{Jensen2014:SubcompleteAndLForcingSingapore}, and shown to be iterable with revised countable support. The main feature of subcomplete forcing that makes it interesting is that subcomplete forcing notions cannot add reals, and as a consequence, \SCFA is compatible with \CH. In fact, Jensen \cite{Jensen:FAandCH} showed that \SCFA is even compatible with $\diamondsuit$, and hence does not imply that the nonstationary ideal on $\omega_1$ is $\omega_2$-saturated. On the other hand, \SCFA does have many of the major consequences of Martin's Maximum, such as the singular cardinal hypothesis. Since \SRP is known to imply that the nonstationary ideal on $\omega_1$ is $\omega_2$-saturated, and that \CH fails, finding a fragment of \SRP for subcomplete forcing was subtle, but in \cite{Fuchs:CanonicalFragmentsOfSRP}, a principle satisfying the two desiderata listed above was found. While the original strong reflection principle can be formulated as postulating that every \emph{projective stationary subset} of $[H_\kappa]^\omega$ contains a continuous $\in$-chain, for regular $\kappa\ge\omega_2$, the subcomplete fragment of \SRP asserts this only for \emph{spread out sets}, and for $\kappa>2^\omega$.

Naturally, there are limitations to the extent to which \ref{item:ImpliesMajorConsequences} can be true. Thus, Larson \cite{Larson:SeparatingSRP} introduced a diagonal version of simultaneous reflection of stationary sets of ordinals, called $\OSR{\omega_2}$, which follows from Martin's Maximum, but not from \SRP. This principle can be generalized to any regular cardinal $\kappa$ greater than $\omega_2$, and it was shown in \cite{Fuchs-LambieHanson:SeparatingDSR} that \SRP does not even imply the weakest versions of these principles, while Fuchs \cite{Fuchs:DiagonalReflection} showed that these principles do follow from \SCFA, as long as $\kappa>2^\omega$. Since $\SC$-$\SRP$ is weaker than $\SRP$, this shows that $\SC$-$\SRP$ does not capture these diagonal reflection principles either, which do follow from \SCFA.

Since these ordinal diagonal reflection principles are underlying the results on the failure of weak square principles under \SCFA shown in \cite{Fuchs:DiagonalReflection}, we push here further in this direction, to find a principle of reflection of generalized stationarity that does capture these consequences of \SCFA/\MM, and that can be relativized to an arbitrary forcing class (resulting in the ``fragments'' of the principle), just like \SRP. We call the resulting principle the \emph{diagonal strong reflection principle,} \DSRP. It unifies both the (relevant fragment) of \SRP and certain diagonal reflection principles the first author introduced in \cite{Cox:DRP}. It also gives rise to some new kinds of exact diagonal reflection principles for sets of ordinals.

For the most part, we will be working with a technical simplification of the notion of subcompleteness, called $\infty$-subcompleteness and introduced in Fuchs-Switzer \cite{Fuchs-Switzer:IterationThmForSubversions}. This leads to a simplification of the adaptation of projective stationarity to the context of this version of subcompleteness. Working with the original notion of subcompleteness adds some technicalities, but does not change much.

The article is organized as follows. In Section \ref{sec:Background}, we will give some background on generalized stationarity, subcomplete forcing and some material from \cite{Fuchs:CanonicalFragmentsOfSRP} on the fragments of \SRP. Then, in Section \ref{sec:SDRP-Gamma}, we will formulate the $\Gamma$-fragment of the diagonal strong reflection principle in full generality, for an arbitrary forcing class $\Gamma$. In the subsequent sections \ref{sec:SSP-DSRP} and \ref{sec:SC-DSRP}, we will treat the cases where $\Gamma$ is the class of all stationary set preserving forcing notions, or the class of all subcomplete forcing notions, respectively, and formulate these principles combinatorially. Here, the notion of a spread out set will make a reappearance, emphasizing its naturalness. Then, in Section \ref{subsec:DSRPconsequences}, we will derive consequences of the principles mentioned above. We divide these consequences in two parts: first, Subsection \ref{subsec:FilteredConsequences} contains consequences that filter through an appropriate version of the diagonal reflection principles of \cite{Cox:DRP}, while Subsection \ref{subsec:Unfiltered} contains some consequences that don't among them some new principles of simultaneous stationary reflection that can be viewed as diagonal reflection principles, enriched with exactness (in a sense to be made explicit).

In Section \ref{sec:DRPlimitations}, we say a few words about limitations of some of the principles under investigation. We separate the diagonal stationary reflection principle from \MM, we show a localized version of this separation for the subcomplete fragment of these principles, and we show that the diagonal reflection principle of \cite{Cox:DRP} does not limit the size of $2^{\omega_1}$.

We close with a few open questions in Section \ref{sec:Questions}.

\section{Some background}
\label{sec:Background}

This section summarizes some definitions and facts we will need. For more detail, we refer to \cite{Fuchs:CanonicalFragmentsOfSRP}. We begin by introducing some notation around generalized stationarity, see see Jech \cite{Jech:StationarySetsHST} for an overview article.

\begin{defn}
\label{defn:ProjectionsAndLiftings}
Let $\kappa$ be a regular cardinal, and let $A\sub\kappa$ be unbounded. Let $\kappa\sub X$. Then
\[\lifting(A,[X]^\omega)=\{x\in[X]^\omega\st\sup(x\cap\kappa)\in A\}\]
is the \emph{lifting} of $A$ to $[X]^\omega$.
Now let $S\sub[X]^\omega$ be stationary. If $W\subseteq X\subseteq Y$, then we define the projections of $S$ to $[Y]^\omega$ and $[W]^\omega$ by
\[S\projectup[Y]^\omega=\{y\in[Y]^\omega\st y\cap X\in S\}\]
and
\[S\projectdown[W]^\omega=\{x\cap W\st x\in S\}.\]
\end{defn}

\begin{defn}
\label{def:Chains}
Let $\kappa$ be a regular uncountable cardinal, and let $S\sub[H_\kappa]^\omega$ be stationary. A \emph{continuous $\in$-chain through $S$ of length $\lambda$} is a sequence $\seq{X_i}{i<\lambda}$ of members of $S$, increasing with respect to $\in$, such that for every limit $j<\lambda$, $X_j=\bigcup_{i<j}X_i$.
\end{defn}

\begin{defn}[{Feng \& Jech \cite{FengJech:ProjectiveStationarityAndSRP}}]
\label{def:ProjectiveStationary}
Let $D$ be a set (usually of the form $H_\kappa$, for some regular uncountable $\kappa$) with $\omega_1\sub D$. Then a set $S\sub[D]^\omega$ with $\bigcup S=D$ is \emph{projective stationary} (in $D$) if for every stationary set $T\sub\omega_1$, the set $\{X\in S\st X\cap\omega_1\in T\}$ is stationary.
\end{defn}

The following is not the original formulation of \SRP due to \Todorcevic, but it was shown by Feng and Jech to be an equivalent way of expressing the principle.

\begin{defn}
\label{def:SRP}
Let $\kappa\ge\omega_2$ be regular. Then the \emph{strong reflection principle at $\kappa$,} denoted $\SRP(\kappa)$, states that whenever $S$ is projective stationary in $H_\kappa$, then there is a continuous 
$\in$-chain of length $\omega_1$ through $S$. The \emph{strong reflection principle} \SRP states that $\SRP(\kappa)$ holds for every regular $\kappa\ge\omega_2$.
\end{defn}

\begin{defn}
\label{def:FA_Gamma}
Let $\Gamma$ be a class of forcing notions. The \emph{forcing axiom for $\Gamma$}, denoted $\FA(\Gamma)$, states that whenever $\P$ is a forcing notion in $\Gamma$ and $\seq{D_i}{i<\omega_1}$ is a sequence of dense subsets of $\P$, there is a filter $F\sub\P$ such that for all $i<\omega_1$, $F\cap D_i\neq\leer$.
\end{defn}

\begin{defn}
\label{def:SSP}
We write $\SSP$ for the class of all forcing notions that preserve stationary subsets of $\omega_1$.
\end{defn}

The principle $\FA(\SSP)$ is known as Martin's Maximum, \MM. The next definition introduces the canonical forcing that can be used to show that Martin's Maximum implies \SRP.

\begin{defn}
\label{def:P_S}
$\P_S$ is the forcing notion consisting of continuous $\in$-chains through $S$ of countable successor length, ordered by end-extension.
\end{defn}

\begin{fact}[Feng \& Jech]
\label{fact:SisProjStatIffP_SisSSP}
Let $\kappa\ge\omega_2$ be an uncountable regular cardinal. Then a stationary set $S\sub[H_\kappa]^\omega$ is projective stationary iff $\P_S\in\SSP$.
\end{fact}

The concept of projective stationarity was generalized in \cite{Fuchs:CanonicalFragmentsOfSRP} as follows.

\begin{defn}
\label{def:GammaProjectiveStationary}
Let $\Gamma$ be a forcing class. Then a stationary subset $S$ of $H_\kappa$, where $\kappa\ge\omega_2$ is regular, is \emph{$\Gamma$-projective stationary} iff $\P_S\in\Gamma$.
\end{defn}

Generalizing the above formulation of \SRP, we arrive at the fragments of this principle, as introduced in \cite{Fuchs:CanonicalFragmentsOfSRP}.

\begin{defn}
\label{def:Gamma-SRP}
Let $\Gamma$ be a forcing class. Let $\kappa\ge\omega_2$ be regular. The \emph{strong reflection principle for $\Gamma$ at $\kappa$}, denoted $\Gamma$-$\SRP(\kappa)$, states that whenever $S\sub[H_\kappa]^\omega$ is $\Gamma$-projective stationary, then $S$ contains a continuous chain of length $\omega_1$. The \emph{strong reflection principle for $\Gamma$,} $\Gamma$-\SRP, states that $\Gamma$-$\SRP(\kappa)$ holds for every $\kappa\ge\omega_2$.
\end{defn}

By design, $\FA(\Gamma)$ implies $\Gamma$-\SRP. Let us now turn to subcompleteness and its simplification, $\infty$-subcompleteness, introduced in \cite{Fuchs-Switzer:IterationThmForSubversions}.

\begin{defn}
\label{def:Full}
A transitive model $N$ of \ZFCm is \emph{full} if there is an ordinal $\gamma>0$ such that $L_\gamma(N)\models\ZFCm$ and $N$ is regular in $L_\gamma(N)$, meaning that if $a\in N$, $f:a\To N$ and $f\in L_\gamma(N)$, then $\ran(f)\in N$. A set $X$ is full if the transitive isomorph of $\kla{X,\in\cap X^2}$ is full.
\end{defn}

\begin{defn}
\label{def:Density}
The density of a poset $\P$, denoted $\delta(\P)$, is the least cardinal $\delta$ such that there is a dense subset of $\P$ of size $\delta$.
\end{defn}

\begin{defn}
\label{def:(ininifty-)subcompleteness}
A forcing notion $\P$ is \emph{subcomplete} if there is a cardinal $\theta$ which \emph{verifies} the subcompleteness of $\P$, which means that $\P\in H_\theta$, and for any \ZFCm{} model $N=L_\tau^A$ with $\theta<\tau$ and $H_\theta\sub N$, any $\sigma:\bN\prec N$ such that $\bN$ is countable, transitive and full and such that $\P,\theta,\eta\in\ran(\sigma)$, any $\bar{G}\sub\bar{\P}$ which is $\bar{\P}$-generic over $\bN$, any $\bar{s}\in\bN$, and any ordinals $\blambda_0,\ldots,\blambda_{n-1}$ such that $\blambda_0=\On\cap\bN$ and $\blambda_1,\ldots,\blambda_{n-1}$ are regular in $\bN$ and greater than $\delta(\bar{\P})^\bN$, the following holds. Letting $\sigma(\kla{\bar{\theta},\bar{\P},\bar{\eta}})=\kla{\theta,\P,\eta}$, and setting $\bar{S}=\kla{\bar{s},\bar{\theta},\bar{\P}}$, there is a condition $p\in\P$ such that whenever $G\sub\P$ is $\P$-generic over $\V$ with $p\in G$, there is in $\V[G]$ a $\sigma'$ such that
	\begin{enumerate}[label=(\arabic*)]
		\item $\sigma':\bN\prec N$,
        \label{item:FirstSubcompletenessCondition}
		\item $\sigma'(\bar{S})=\sigma(\bar{S})$,
		\item $(\sigma')``\bar{G}\sub G$,
		\item
		\label{item:SupremumCondition}
        $\sup\sigma``\blambda_i=\sup\sigma'``\blambda_i$ for each $i<n$.
	\end{enumerate}
$\P$ is \emph{$\infty$-subcomplete} iff the above holds, with condition \ref{item:SupremumCondition} removed.

We denote the classes of subcomplete and $\infty$-subcomplete forcing notions by $\SC$ and $\infSC$, respectively.
\end{defn}

The following definition, again from \cite{Fuchs:CanonicalFragmentsOfSRP}, is designed to capture $\infSC$-projective stationarity.

\begin{defn}
\label{def:SpreadOut}
Let $D$ be a set (usually of the form $D=H_\kappa$, for some uncountable regular cardinal $\kappa$). A set $S\sub[D]^\omega$ with $\bigcup S=D$ is \emph{spread out} (in $D$) if for every sufficiently large cardinal $\theta$ with $S\in H_\theta$, whenever $\tau$, $A$, $X$ and $a$ are such that $H_\theta\sub L_\tau^A=N\models\ZFCm$, $\theta<\tau$, $S,a,\theta\in X$, $N|X\prec N$, and $N|X$ is countable and full, then there are a $Y$ such that $N|Y\prec N$ and an isomorphism $\pi:N|X\To N|Y$ such that $\pi(a)=a$  and $Y\cap H_\kappa\in S$.
\end{defn}

The remaining definitions and results are from \cite{Fuchs:CanonicalFragmentsOfSRP}.

\begin{defn}
\label{def:WeaklySpreadOut}
Let $D$ be a set. A set $S\sub[D]^\omega$ with $\bigcup S=D$ is \emph{weakly spread out} if there is a set $b$ such that the condition described in Definition \ref{def:SpreadOut} is true of all $X$ with $S,\theta,b\in X$.
\end{defn}

\begin{fact}
\label{fact:WeaklySpreadOutImpliesSpreadOut}
Let $\kappa$ be an uncountable regular cardinal. A stationary set $S\sub[H_\kappa]^\omega$ is spread out iff it is weakly spread out.
\end{fact}

The following theorem is the analog of Fact \ref{fact:SisProjStatIffP_SisSSP} for $\infty$-subcompleteness, giving us a combinatorial characterization of $\infSC$-projective stationarity.

\begin{thm}
\label{thm:SpreadOutIffinfSCProjectiveStationary}
Let $\kappa$ be an uncountable regular cardinal, and let $S\sub[H_\kappa]^\omega$. Then $S$ is spread out iff $S$ is \infSC-projective stationary.
\end{thm}

Spread out sets are stationary, and in fact projective stationary.

\begin{obs}[{\cite[Obs.~2.28]{Fuchs:CanonicalFragmentsOfSRP}}]
\label{obs:SpreadOutImpliesProjectiveStationary}
If a set $S\sub[D]^\omega$ is spread out in $D$, with $\omega_1\sub D$, then $S$ is projective stationary in $D$.	
\end{obs}

Spread out sets satisfy some natural closure properties.

\begin{obs}
\label{obs:SpreadOutSetsClosedUnderIntersectionWithClubs}
Let $\kappa$ be an uncountable regular cardinal, let $S\sub[H_\kappa]^\omega$ be spread out, and let $C\sub[H_\kappa]^\omega$ be club. Then $S\cap C$ is spread out.
\end{obs}

\begin{obs}
Let $A\sub B\sub C$, and suppose $S$ is spread out in $B$. Then both $S\projectdown[A]^\omega$ and $S\projectup[C]^\omega$ are spread out.
\end{obs}

The natural analogs of these closure properties are known to hold for projective stationary sets as well. We will use the following standard notation frequently.

\begin{defn}
Let $\kappa$ be an ordinal, and let $\rho$ be a regular cardinal. Then we write
\[S^\kappa_\rho=\{\alpha<\kappa\st\cf(\kappa)=\rho\}.\]	
\end{defn}

The following provides an important collection of spread out sets.

\begin{lem}
\label{lem:LiftingsAreSpreadOut}
Let $\kappa>2^\omega$ be a regular cardinal, and let $B\sub S^\kappa_\omega$ be stationary. Then the set
\[S=\{X\in[H_\kappa]^\omega\st \sup(X\cap\kappa)\in B\}=\lifting(B,[H_\kappa]^\omega)\]
is spread out.
\end{lem}

\section{The diagonal strong reflection principle for a forcing class}
\label{sec:SDRP-Gamma}


The idea for the diagonal strong reflection principle is that instead of guaranteeing the existence of a continous $\in$-chain of length $\omega_1$ through each projective stationary set individually, it postulates the existence of such a sequence through a whole \emph{collection} $\mathcal{S}$ of (appropriate) sets. The way the sequence passes through the sets is designed so as to give it a ``diagonal'' flavor. The following definition makes this precise.

\begin{defn}
\label{def:DSRP-preliminiary}
Let $\mathcal{S}$ be a collection of stationary subsets of $[H_\kappa]^\omega$. Let $\vT=\seq{T_i}{i<\omega_1}$ be a sequence of pairwise disjoint stationary subsets of $\omega_1$, and let $X$ be a set. Then $\kla{\vQ,\vS}$ is a \emph{diagonal chain through $\mathcal{S}$ up to $X$ with respect to $\vT$} if
\begin{enumerate}
  \item $\vQ=\seq{Q_i}{i<\omega_1}$ is a continuous $\in$-chain of countable subsets of $H_\kappa$: 
  	 \begin{enumerate}
  	 \item  for all $i<\omega_1$, $Q_i\in Q_{i+1}$,
  	 \item and for limit $\lambda<\omega_1$, $Q_\lambda=\bigcup_{i<\lambda}Q_i$,
  	 \end{enumerate}
  \item $\vS=\seq{S_i}{i<\omega_1}$ is a sequence of members of $\mathcal{S}$, such that whenever $i\in T_j$, then $Q_i\in S_j$,
  \item $H_\kappa\cap X=\bigcup_{\alpha<\omega_1}Q_\alpha$, and for all $\alpha<\omega_1$, $\seq{Q_i}{i<\alpha}\in H_\kappa\cap X$,
  \item $\mathcal{S}\cap X=\{S_i\st i<\omega_1\}$.
\end{enumerate}

We also formulate a slightly simpler version of this concept, independent of the particular sequence $\vT$. All we need is $\mathcal{S}$, a collection of stationary subsets of $[H_\kappa]^\omega$. Then $\seq{Q_i}{i<\omega_1}$ is a \emph{diagonal chain through $\mathcal{S}$ up to $X$} if:
\begin{enumerate}
  \item $\vQ=\seq{Q_i}{i<\omega_1}$ is a continuous $\in$-chain of countable subsets of $H_\kappa$, %
  \item For every $S\in X\cap\mathcal{S}$, the set $\{i<\omega_1\st Q_i\in S\}$ is stationary in $\omega_1$,
  \item $H_\kappa\cap X=\bigcup_{i<\omega_1}Q_i$, and for all $\alpha<\omega_1$, $\seq{Q_i}{i<\alpha}\in H_\kappa\cap X$.
\end{enumerate}
Such a chain is \emph{exact} if in addition,
\begin{enumerate}
\setcounter{enumi}{3}
  \item for every $i<\omega_1$, $Q_i\in\bigcup(X\cap\mathcal{S})$.
\end{enumerate}
\end{defn}

\begin{obs}
Let $\mathcal{S}$, $\kappa$, $\vT$ be as in Definition \ref{def:DSRP-preliminiary}, and suppose that $\kla{\vQ,\vS}$ is a diagonal chain through $\mathcal{S}$ up to $X$ with respect to $\vT$. Then
\begin{enumerate}[label=\rm{(\arabic*)}]
  \item $\vQ$ is a diagonal chain through $\mathcal{S}$ up to $X$.
  \item If $\bigcup_{i<\omega_1}T_i=\omega_1$, then $\vQ$ is an exact diagonal chain through $\mathcal{S}$ up to $X$.
  \item
  \label{item:ContainingAClubIsEnoughForExactness}
  If $\kappa\ge\omega_2$ is regular, $\kla{H_\kappa\cap X,\in}\prec\kla{H_\kappa,\in}$, $\vT\in X$ and $\bigcup_{i<\omega_1}T_i$ contains a club, then there is a diagonal chain $\kla{\vR,\vS}$ through $\mathcal{S}$ up to $X$ with respect to some $\seq{\bar{T}_i}{i<\omega_1}$ such that $\bigcup\bar{T}_i=\omega_1$. Hence, $\vR$ is an exact diagonal chain through $\mathcal{S}$ up to $X$.
\end{enumerate}
\end{obs}

\begin{proof}
We outline the straightforward proof of \ref{item:ContainingAClubIsEnoughForExactness}. By elementarity of $H_\kappa\cap X$, and since $\vT\in X$, it follows that there is a club $C\sub\bigcup_{i<\omega_1}T_i$ in $X$. Hence, the monotone enumeration $f$ of $C$ is also in $X$. Define for $i<\omega_1$:
\[R_i=Q_{f(i)},\ \bar{T}_i=f^{-1}``T_i.\]
It is then easy to check that $\vec{\bT}$ is a partition of $\omega_1$ into stationary sets and $\kla{\vR,\vS}$ is a diagonal chain through $\mathcal{S}$ with respect to $\vec{\bT}$, as wished. Since $f\in X$ and for all $\alpha<\omega_1$, $\vQ\rest\alpha\in X$, it follows that for all $\alpha<\omega_1$, $\vR\rest\alpha\in X$, as $\kla{H_\kappa\cap X,\in}\prec\kla{H_\kappa,\in}$.    	
\end{proof}

We introduce a canonical forcing to add diagonal chains. It is a variation of a forcing notion from Cox \cite{Cox:DRP}, which, in turn, is based on a poset defined by Foreman \cite{Foreman:SmokeAndMirrors}.

\begin{defn}
\label{def:ForcingDSRP}
Let $\kappa \ge \omega_2$ be regular, $\vec{T} = \langle T_i \ : \ i < \omega_1 \rangle$ be sequence of pairwise disjoint stationary subsets of $\omega_1$, and let $\mathcal{S}$ be a nonempty collection of stationary subsets of $[H_\kappa]^\omega$. 
The poset $\mathbb{P}^{\DSRP}_{\mathcal{S},\vT}$ consists of conditions of the form
\[
p = \kla{\vec{Q}^p,\vec{S}^p}
\]
where, for some $\delta^p,\lambda^p<\omega_1$:
\begin{enumerate}
 \item $\vec{Q}^p=\seq{Q^p_\alpha}{\alpha\le\delta^p}$ is a continuous $\in$-chain of elements of $[H_\kappa]^\omega$. 
 \item $\vec{S}^p=\seq{S^p_i}{i<\lambda^p}$ is a sequence such that for every $i<\lambda^p$, $S^p_i\in\mathcal{S}$.
 \item Whenever $\alpha\le\delta^p$ and $i<\omega_1$ are such that $\alpha\in T_i$, then $i<\lambda^p$ and $Q^p_\alpha\in S^p_i$.
\end{enumerate}
The ordering is by extension of functions in both coordinates.
\end{defn}

Let us note some basic properties of this forcing notion.

\begin{fact}
\label{fact:PropertiesOfP-DSRP}
Let $\kappa$ be an uncountable regular cardinal, $\leer\neq\mathcal{S}\sub\power([H_\kappa]^\omega)$ a collection of stationary subsets, and $\vT$ an $\omega_1$-sequence of pairwise disjoint stationary subsets of $\omega_1$, and let $\P=\P^\DSRP_{\mathcal{S},\vT}$. Then
\begin{enumerate}[label=(\arabic*)]
  \item
  \label{item:GenericSequenceHasLengthOmega1-DSRP}
  for every countable ordinal $\gamma$, the set of conditions $p$ with $\delta^p,\lambda^p\ge\gamma$ is dense in $\P_S$,
  \item
  \label{item:UnionOfGenericIsExhaustiveInFirstCoordinate}
  for every $a\in H_\kappa$, the set of conditions $p$ such that there is an $\alpha\le\delta^p$ with $a\in Q^p_\alpha$ is dense in $\P$,
  \item
  \label{item:UnionOfGenericIsExhaustiveInSecondCoordinate}
  for every $S\in\mathcal{S}$, the set of conditions $p$ such that there is an $i<\lambda^p$ such that $S^p_i=S$ is dense.
\end{enumerate}
\end{fact}

\begin{proof}
We need some facts before being able to prove this. The first fact is a generalization of a result from \cite{HFriedman:OnClosedSetsOfOrdinals}.

\medskip

\noindent\emph{Fact 1:} Let $\seq{A_i}{i<\omega_1}$ be a sequence of stationary subsets of $\omega_1$, and let $t:\omega_1\To\omega_1$ be a function. Then, for any $\beta,\alpha<\omega_1$, with $\alpha>0$, there is a normal function $f:\alpha\To\omega_1$ such that for all $\xi<\alpha$, $f(\xi)\in A_{t(\xi)}$ and $f(0)>\beta$.

\begin{proof}[Proof of Fact 1:]
Let $\gamma,\alpha,\sigma$ be countable ordinals, $\alpha$ a limit. Say that $\gamma$ is $\alpha$-approachable from $\sigma$ if for every $\delta<\gamma$, there is a normal function $f:[\sigma,\sigma+\alpha]\To[\delta,\gamma]$ such that for all $\xi\in[\sigma,\sigma+\alpha]$, $f(\xi)\in A_{t(\xi)}$ and $f(\sigma+\alpha)=\gamma$. We refer to such a function as a nice function from $[\sigma,\sigma+\alpha]$ to $[\delta,\gamma]$.

We will prove by induction on limit ordinals $\alpha<\omega_1$: for every $\sigma<\alpha$, the set of $\lambda<\omega_1$ such that $\lambda$ is $\alpha$-approachable from $\sigma$ is unbounded in $\omega_1$. This clearly proves Fact 1.

If $\alpha=\omega$, then let $\sigma<\omega_1$ be given. Fixing any $\beta<\omega_1$, we have to find a countable $\lambda\ge\beta$ that is $\alpha$-approachable from $\sigma$. To this end, let
\[\lambda\in \left(A_{t(\sigma+\omega)}\cap\bigcap_{\sigma\le\xi<\sigma+\omega}\Lim(A_{t(\xi)})\right)\ohne\beta.\]
Given any $\delta<\lambda$, it is then easy to define $f:[\sigma,\sigma+\omega]\To[\beta,\lambda]$ recursively so that $f$ is strictly increasing, $f(0)>\delta$, for $\xi<\omega$, $f(\xi)\in A_{t(\xi)}$, and $\sup\{f(\sigma+n)\st n<\omega\}=\lambda$. Thus, setting $f(\sigma+\omega)=\lambda$ yields a nice function from $[\sigma,\sigma+\omega]$ to $[\delta,\lambda]$, as wished.

Now suppose this has been proven for $\alpha$. We have to show the claim for $\alpha+\omega$. To this end, fix $\sigma<\omega_1$. Given an arbitrary $\beta<\omega_1$, we have to find a countable $\lambda\ge\beta$ which is $\alpha+\omega$-approachable from $\sigma$. Let $D=\{\gamma<\omega_1\st\gamma\ \text{is $\alpha$-approachable from $\sigma$}\}$. Inductively, this set is unbounded in $\omega_1$. Let
\[\lambda\in\left(A_{t(\sigma+\alpha+\omega)}\cap\Lim(D)\cap\bigcap_{\sigma+\alpha\le\xi<\sigma+\alpha+\omega}\Lim(A_{t(\xi)})\right)\ohne\beta.\]
To see that $\lambda$ is $\alpha+\omega$-approachable from $\sigma$, let $\delta<\lambda$. Let $\gamma\in (D\cap\lambda)\ohne(\delta+1)$. Since $\gamma$ is $\alpha$-approachable from $\sigma$, there is a nice function $\barf$ from $[\sigma,\sigma+\alpha]$ to $[\delta,\gamma]$. As in the case $\alpha=\omega$, we can extend $\barf$ to a normal and cofinal function $\barf':[\sigma,\sigma+\alpha+\omega)\To\lambda$, such that for each $n<\omega$, $\barf'(\sigma+\alpha+n)\in A_{t(\sigma+\alpha+n)}$. Since $\lambda=\sup\ran(\barf')$ and $\lambda\in A_{t(\sigma+\alpha+\omega)}$, we can extend $\barf'$ to a nice function from $[\sigma,\sigma+\alpha+\omega]$ to $[\delta,\lambda]$ by specifying that $f(\sigma+\alpha+\omega)=\lambda$.

Finally, suppose $\alpha$ is a limit of limit ordinals, and the claim has been proven for all limit ordinals below $\alpha$. Fixing $\sigma,\beta<\omega_1$, we have to find a countable $\lambda>\beta$ which is $\alpha$-approachable from $\sigma$. Let $\seq{s_n}{n<\omega}$ be increasing and cofinal in $\alpha$, $s_0=0$. Let $\sigma_n=\sigma+s_n$. Let
\[D_n=\{\gamma<\omega_1\st\gamma\ \text{is $(s_{n+1}-s_n)$-approachable from $\sigma_n$}\}.\]
Inductively, $D_n$ is unbounded in $\omega_1$, for each $n<\omega$. Let
\[\lambda\in\left(A_{t(\sigma+\alpha)}\cap\bigcap_{n<\omega}\Lim(D_n)\right)\ohne\beta.\]
To see that $\lambda$ is $\alpha$-approachable from $\sigma$, fix $\delta<\lambda$. Find an increasing sequence $\seq{\delta_n}{n<\omega}$ cofinal in $\lambda$ such that $\delta_0>\delta$ and $\delta_n\in D_n$, for all $n<\omega$. Using the definition of $D_n$, we can now find a sequence of functions $\seq{f_n}{n<\omega}$ such that
\begin{itemize}
  \item $f_n$ is a nice function from $[\sigma_n,\sigma_{n+1}]$ to $[\delta_n,\delta_{n+1}]$.
  \item $f_{n+1}(\sigma_{n+1})=\delta_{n+1}=f_n(\sigma_{n+1})$.
\end{itemize}
Thus, the union $\bigcup_{n<\omega}f_n$ is a function that can be extended to a nice function from $[\sigma,\sigma+\alpha]$ to $[\delta,\lambda]$ by mapping $\sigma+\alpha$ to $\lambda$.
\end{proof}

\noindent\emph{Fact 2:} If $\seq{S_i}{i<\omega_1}$ is a sequence of stationary subsets of $[H_\kappa]^\omega$ and $\seq{T_i}{i<\omega_1}$ is a sequence of pairwise disjoint stationary subsets of $\omega_1$, then
for any $\gamma<\omega_1$, there is a continuous $\in$-chain $\seq{Q_\alpha}{\alpha<\gamma}$ such that for all $\alpha<\gamma$, if $\alpha\in T_i$, then $Q_\alpha\in S_i$.

\begin{proof}[Proof of Fact 2:] This is a strengthening of \cite[Lemma 1.2]{FengJech:ProjectiveStationarityAndSRP}, and the proof of that lemma can be adapted to the present situation.
Let $\Q$ be the forcing to add a continuous $\in$-chain of countable subsets of $H_\kappa$, of length $\omega_1$, by initial segments of successor length. This forcing is $\sigma$-closed. Let $\seq{Q_\alpha}{\alpha<\omega_1}$ be a sequence added by $\Q$, i.e., $\vec{Q}=\bigcup G$, for some $\Q$-generic $G$. In $V[G]$, every $S_i$ is still stationary, so the set $A_i=\{\alpha<\omega_1\st Q_\alpha\in S_i\}$ is stationary in $V[G]$. So by Fact 1, applied in $V[G]$ to the function $t:\omega_1\To\omega_1$ defined by
\[t(\xi)=\left\{
\begin{array}{l@{\qquad}l}
i & \text{if $\xi\in T_i$,}\\
0 & \text{if $\xi\in\omega_1\ohne\bigcup_{j<\omega_1}T_j$.}
\end{array}
\right.\]
there is a normal function $f:\gamma\To\omega_1$ such that for all $\alpha<\gamma$, $f(\alpha)\in A_{t(\alpha)}$, which means that if $\alpha\in T_i$, then $f(\alpha)\in A_i$, and this means that $Q_{f(\alpha)}\in S_i$. So, the sequence $\vec{Q}'=\seq{Q_{f(\alpha)}}{\alpha<\gamma}$ is as wished, and it belongs to $V$, since $\Q$ is countably distributive.
\end{proof}

We can now prove clauses \ref{item:GenericSequenceHasLengthOmega1-DSRP}, \ref{item:UnionOfGenericIsExhaustiveInFirstCoordinate} and \ref{item:UnionOfGenericIsExhaustiveInSecondCoordinate} simultaneously. Fix a condition $p\in\P_S$, and let $\gamma<\omega_1$, $a\in H_\kappa$ and $S\in\mathcal{S}$ be given. We may assume that $\gamma>\delta^p$. We may also assume that $\delta^p+1\in\bigcup_{i<\omega_1}T_i$, for if not, then we may just define $\vT'$ to be like $\vT$, except that $\delta^p+1\in T'_0$, say. Then $p\in\P^\DSRP_{\mathcal{S},\vT'}$, and an extension of $p$ in $\P^\DSRP_{\mathcal{S},\vT'}$ with the desired properties is also an extension of $p$ in $\P^\DSRP_{\mathcal{S},\vT}$. So let's let $i_0$ be such that $\delta^p+1\in T_{i_0}$.

Since it is trivial to extend the second coordinate of a condition, we may assume that $\lambda^p>\gamma$, that for every $\alpha\le\gamma$, if $i<\omega_1$ is such that $\alpha\in T_i$, then $i<\lambda^p$, and that there is some $i<\lambda^p$ such that $S^p_i=S$, taking care of clause \ref{item:UnionOfGenericIsExhaustiveInSecondCoordinate}. In order to be able to use Fact 2 now, we have to perform a little index translation, shifting by $\delta^p+1$. Thus, let's define a sequence $\vT'=\seq{T'_i}{i<\omega_1}$ by letting $T'_i=\{\xi<\omega_1\st(\delta^p+1)+\xi\in T_i\}$. Let's also define $\vS=\seq{S_i}{i<\omega_1}$ by
\[S_i=\left\{
\begin{array}{l@{\qquad}l}
S^p_i&\text{if}\ i<\lambda^p,\ i\neq i_0,\\
\{x\in S^p_{i_0}\st \{a, Q^p_{\delta^p}\}\sub x\} & \text{if}\ i=i_0,\\
S^p_0&\text{if}\ i\ge\lambda^p.
\end{array}
\right.
\]
Clearly, $\vS$ is a sequence of stationary subsets of $[H_\kappa]^\omega$, and $\vT'$ is a sequence of pairwise disjoint stationary subsets of $\omega_1$, so we may apply Fact 2 to give us a continuous $\in$-chain $\seq{R_\xi}{\xi<\bgamma}$, where $\bgamma=(\gamma+1)-(\delta^p+1)$, such that for all $i<\bgamma$, if $i\in T'_j$, then $R_i\in S_j$. The condition $q=\kla{\vQ^p\verl\vR,\vS^p}$ is then an extension of $p$ with all the desired properties. This is because for $i<\bgamma$, if $\delta^p+1+i\in T_j$, then $i\in T'_j$, so that $Q^q_{\delta^p+1+i}=R_i\in S_j=S^q_j$, and in particular, $Q^q_{\delta^p}=Q^p_{\delta^p}\in Q^q_{\delta^p+1}$, as $Q^q_{\delta^p+1}\in S_{i_0}$.
\end{proof}

The following lemma is an immediate consequence of Fact \ref{fact:PropertiesOfP-DSRP}.

\begin{lem}
\label{lem:PropertiesOfTheGeneric}
Let $G$ be generic for $\mathbb{P}^{\DSRP}_{\mathcal{S},\vec{T}}$, where $\mathcal{S}$ is a nonempty collection of stationary subsets of $[H_\kappa]^\omega$ and $\vT$ is an $\omega_1$-sequence of pairwise disjoint stationary subsets of $\omega_1$.
Let $\vec{Q}=\bigcup_{p\in G}\vQ^p$ and $\vS=\bigcup_{p\in G}\vS^p$. Then
\begin{enumerate}
  \item $\vec{Q}$ is a continuous $\in$-chain of length $\omega_1^V$ whose union is $H_\kappa^V$.
  \item $\vS$ is a sequence of length $\omega_1^V$, and $\mathcal{S}=\{S_i\st i<\omega_1^V\}$.
  \item For all $i<\omega_1$ and all $\alpha\in T_i$, we have that $Q_\alpha\in S_i$.
\end{enumerate}
\end{lem}

We should now define the instances of the diagonal strong reflection principle.

\begin{defn}
\label{def:DSRPofSAndT}
Let $\mathcal{S}$ be a collection of stationary subsets of $[H_\kappa]^\omega$, where $\kappa>\omega_1$ is regular, $\vT$ is an $\omega_1$-sequence of pairwise disjoint stationary subsets of $\omega_1$, and $\theta$ a sufficiently large regular cardinal (so that $\mathcal{S}\sub H_\theta$, that is, $\theta>2^{{<}\kappa}$). Then the \emph{diagonal strong reflection principle for $\kla{\mathcal{S},\vT}$,} $\DSRP(\mathcal{S},\vT)$, says that
\[\{X\in[H_\theta]^{\omega_1}\st\omega_1\sub X\ \text{and there is a diagonal chain through $\mathcal{S}$ up to $X$ wrt.~$\vT$}\}\]
is stationary in $H_\theta$.

The \emph{diagonal strong reflection principle for $\mathcal{S}$}, $\DSRP(\mathcal{S})$, says that
\[\{X\in[H_\theta]^{\omega_1}\st\omega_1\sub X\ \text{and there is a diagonal chain through $\mathcal{S}$ up to $X$}\}\]
is stationary in $H_\theta$.

The \emph{exact diagonal strong reflection principle for $\mathcal{S}$}, $\eDSRP(\mathcal{S})$, says that
\[\{X\in[H_\theta]^{\omega_1}\st\omega_1\sub X\ \text{and there is an exact diagonal chain through $\mathcal{S}$ up to $X$}\}\]
is stationary in $H_\theta$.
\end{defn}

\begin{remark}
\label{remark:GettingExactDSRPwhenTisTotal}
Let $\mathcal{S}$, $\kappa$ and $\vT$ be as in Definition \ref{def:DSRPofSAndT}. Then we have the following implications:
\begin{enumerate}
\item $\DSRP(\mathcal{S},\vT)\implies\DSRP(\mathcal{S})$
\item If $\bigcup_{i<\omega_1}T_i=\omega_1$, then $\DSRP(\mathcal{S},\vT)\implies\eDSRP(\mathcal{S})$
\end{enumerate}
\end{remark}

\begin{lem}
\label{lem:FA(P-DSRP)impliesDSRPatSAndT}
Let $\mathcal{S}$ be a nonempty collection of stationary subsets of $[H_\kappa]^\omega$, and let $\vT$ be an $\omega_1$-sequence of pairwise disjoint stationary subsets of $\omega_1$ such that $\FA(\{\P^{\DSRP}_{\mathcal{S},\vT}\})$ holds. Then $\DSRP(\mathcal{S},\vT)$ holds.
\end{lem}

\begin{proof} Let $\P=\P^\DSRP_{\mathcal{S},\vT}$. Let $\theta$ be a sufficiently large regular cardinal, and let $\mathcal{A}=\kla{H_\theta,\in,\P,\mathcal{S},\vT,F,<^*}$, where $F$ is some function from $H_\theta^{{<}\omega}$ to $H_\theta$ and $<^*$ is a well-order of $H_\theta$. Let
\[A=\{X\in[H_\theta]^{\omega_1}\st\omega_1\sub X\ \text{and there is a diagonal chain through $\mathcal{S}$ up to $X$}\}.\]
To show that $A$ is stationary, it suffices to show that there is an $X\in A$ such that $\mathcal{A}|X\prec\mathcal{A}$ and $X\cap\omega_2\in\omega_2$; see \cite[Exercise 38.10]{ST3}. By the argument of the proof of \cite[Lemma 2.53]{Woodin:ADforcingAxiomsNonstationaryIdeal}, it follows from $\FA(\{\P\})$ that there is an $X\in[H_\theta]^{\omega_1}$ with $\omega_1\sub X$ and $\mathcal{A}|X\prec\mathcal{A}$, such that there is a $G$ which is ($X,\P$)-generic (meaning that $G$ is a filter in $\P$ such that for every dense subset $D\sub\P$ with $D\in X$, $G\cap D\cap X\neq\leer$). Since $\omega_1\sub X$, it follows that $X\cap\omega_2\in\omega_2$. To see that $X\in A$, let
let $\vec{Q}=\bigcup_{p\in G}\vQ^p$ and $\vS=\bigcup_{p\in G}\vS^p$. It then follows from Lemma \ref{lem:PropertiesOfTheGeneric} that $\kla{\vQ,\vS}$ is a diagonal chain through $\kla{X,\mathcal{S}}$ with respect to $\vT$. To see that for $\alpha<\omega_1$, $\vec{Q}\rest\alpha=\seq{Q_i}{i<\alpha}\in H_\kappa\cap X$, note that the set $D_\alpha$ of conditions $p\in\P$ with $\delta^p\ge\alpha$ is dense in $\P$ and an element of $X$. Hence, there is a $p\in G\cap D_\alpha\cap X$. The restriction of the first component of $p$ to $\alpha$ is then also in $X$, and it is the sequence $\vec{Q}\rest\alpha$.
\end{proof}

\begin{defn}
\label{def:Gamma-projectivestationaryForDSRP}
Let $\vec{T}$ be an $\omega_1$-sequence of pairwise disjoint stationary subsets of $\omega_1$, let $\kappa>\omega_1$ be regular, and let $\mathcal{S}\sub\power([H_\kappa]^\omega)$ be a nonempty family of stationary sets. Let $\Gamma$ be a forcing class. Then we say that $\kla{\mathcal{S},\vec{T}}$ is \emph{$\Gamma$-projective stationary} if $\P^{\DSRP}_{\mathcal{S},\vec{T}}\in\Gamma$. We say that $\mathcal{S}$ is \emph{$\Gamma$-projective stationary} if there is a $\vec{T}$ such that $\kla{\mathcal{S},\vec{T}}$ is $\Gamma$-projective stationary.
\end{defn}

Motivated by Lemma \ref{lem:FA(P-DSRP)impliesDSRPatSAndT}, it thus makes sense to define:

\begin{defn}
\label{def:Gamma-DSRP}
For a forcing class $\Gamma$ and a regular $\kappa>\omega_1$, the \emph{$\Gamma$-fragment of the diagonal strong reflection principle at $\kappa$}, $\Gamma$-$\DSRP(\kappa)$, says that whenever $\mathcal{S}$ is a collection of stationary subsets of $[H_\kappa]^\omega$ and $\vT$ is an $\omega_1$-sequence of pairwise disjoint stationary subsets of $\omega_1$ such that $\kla{\mathcal{S},\vT}$ is $\Gamma$-projective stationary, then $\DSRP(\mathcal{S},\vT)$ holds.



And generally, 
$\Gamma$-\DSRP says that
$\Gamma$-$\DSRP(\kappa)$ holds for every regular $\kappa>\omega_1$.

If $\Gamma=\SSP$, then we may omit mention of $\Gamma$. 
\end{defn}


Another way to express Lemma \ref{lem:FA(P-DSRP)impliesDSRPatSAndT} is as follows:

\begin{lem}
\label{lem:FA(Gamma)impliesGamma-DSRP}
Let $\Gamma$ be a forcing class. Then $\FA(\Gamma)$ implies $\Gamma$-\DSRP.
\end{lem}

\section{The stationary set preserving fragment of the diagonal strong reflection principle}
\label{sec:SSP-DSRP}

The appeal of the principle $\Gamma$-\DSRP is that it can be formulated in a way that's purely combinatorial and does not directly refer to the forcing class $\Gamma$ in the cases of main interest to us. We treat the case where $\Gamma$ is the class of all stationary set preserving forcing notions in the present section. Thus, we have to analyze which pairs $\kla{\mathcal{S},\vec{T}}$ are $\SSP$-projective stationary, and to this end, we will employ a few definitions.

\begin{defn}
\label{def:ProjectiveStationaryOnA}
Let $\kappa>\omega_1$ be regular, $S\sub[H_\kappa]^\omega$, and $A\sub\omega_1$. Then $S$ is \emph{projective stationary on $A$} if for every set $B\sub A$ that is stationary in $\omega_1$, the set $\{M\in S\st M\cap\omega_1\in B\}$ is a stationary subset of $[H_\kappa]^\omega$.
\end{defn}

\begin{remark}
\label{rem:VacuousUnlessStationary}
Note that projective stationarity on $A$ is vacuous unless $A$ is a stationary subset of $\omega_1$.
\end{remark}

The following definition is designed to capture $\SSP$-projective stationarity of pairs $\kla{\mathcal{S},\vT}$.

\begin{defn}
\label{def:ProjectiveStationaryOnVecT}
Let $\kappa>\omega_1$ be regular, and let $\mathcal{S}\sub\power([H_\kappa]^\omega)$ be a nonempty collection of stationary subsets of $[H_\kappa]^\omega$. Let $\vec{T}$ be a sequence of pairwise disjoint stationary subsets of $\omega_1$. Then $\mathcal{S}$ is \emph{projective stationary on $\vec{T}$} if the following hold:
\begin{enumerate}[label=(\alph*)]
\item
\label{item:PSonEachTi}
for every $i<\omega_1$, and for every $S\in\mathcal{S}$, $S$ is projective stationary on $T_i$,
\item
\label{item:StrangeCondition}
$\bigcup\mathcal{S}$ is projective stationary on $\{\alpha\in\bigcup_{i<\omega_1}T_i\st\forall\beta<\alpha\quad\alpha\notin T_\beta\}$.
\end{enumerate}
\end{defn}

Note that clause \ref{item:StrangeCondition} can be expressed as saying that $\bigcup\mathcal{S}$ is projective stationary on $\bigcup_{i<\omega_1}T_i\setminus\bigtriangledown_{i<\omega_1}T_i$, and is vacuous if this set is nonstationary (see Remark \ref{rem:VacuousUnlessStationary}). Let's say that $\vT$ is \emph{maximal} in this case. This is equivalent to saying that for every stationary subset $A\sub\bigcup_{i<\omega_1}T_i$, there is an $i<\omega_1$ such that $A\cap T_i$ is stationary. In fact, maximality simplifies the whole concept considerably.

\begin{remark}
\label{rem:ProjStatOnVecTifTisMaximal}
If $\kappa>\omega$ is regular, $\vec{T}$ is an $\omega_1$-sequence of pairwise disjoint stationary subsets of $\omega_1$ that is maximal, and $\mathcal{S}$ is a collection of stationary subsets of $[H_\kappa]^\omega$, then $\mathcal{S}$ is projective stationary on $\vec{T}$ iff every $S\in\mathcal{S}$ is projective stationary on $D=\bigcup_{i<\omega_1}D_i$.

Thus, if $\vT$ is a maximal partition of $\omega_1$ into stationary sets, then $\mathcal{S}$ is projective stationary on $\vT$ iff every $S\in\mathcal{S}$ is projective stationary.
\end{remark}

\begin{proof}
For the direction from left to right, if $A$ is a stationary subset of $D$, then by maximality of $\vec{T}$, there is an $i<\omega_1$ such that $A\cap T_i$ is stationary, so that condition \ref{item:PSonEachTi} of Definition \ref{def:ProjectiveStationaryOnVecT} implies that $\{M\in S\st M\cap\omega_1\in A\}$ is stationary, for every $S\in\mathcal{S}$. Vice versa, if $\mathcal{S}$ is projective stationary on $D$, then condition \ref{item:PSonEachTi} of Definition \ref{def:ProjectiveStationaryOnVecT} follows immediately, and by the remark above, condition \ref{item:StrangeCondition} is vacuous by the maximality of $\vec{T}$.
\end{proof}

Maximal partitions always exist (see \cite[Remark 3.17]{Fuchs:CanonicalFragmentsOfSRP}), and we don't have a use for nonmaximal ones, so the reader may think of this special case in what follows with no loss. Nevertheless, we carry out the analysis in the more general setting.

The assumptions of the following lemma could be weakened, but the present form suffices for our purposes.

\begin{lem}
\label{lem:CountableDistributivityOfP-DSRP}
Let $\kappa$ be an uncountable regular cardinal, $\leer\neq\mathcal{S}\sub\power([H_\kappa]^\omega)$ a collection of stationary subsets, and $\vT$ an $\omega_1$-sequence of pairwise disjoint stationary subsets of $\omega_1$, and let $\P=\P^\DSRP_{\mathcal{S},\vT}$. If every $S\in\mathcal{S}$ is projective stationary on $T_0$, then $\P$ is countably distributive.
\end{lem}

\begin{proof}
We have to show that, given a sequence $\vec{D}=\seq{D_n}{n<\omega}$ of dense open subsets of $\P$, the intersection $\Delta=\bigcap_{n<\omega}D_n$ is dense in $\P$. So, fixing a condition $p\in\P$, we have to find a $q\le p$ in $\Delta$. We may assume that $S=S^p_0$ is defined.

Let $\lambda$ be a regular cardinal much greater than $\kappa$, say $\lambda>2^{2^{\card{\P}}}$, and consider the model $\mathcal{N}=\kla{H_\lambda,\in,<^*,\P,\vec{D},p}$, where $<^*$ is a well-ordering of $H_\lambda$.

Since $S'=\{X\in S\st X\cap\omega_1\in T_0\}$ is stationary, we can let
$\calM=\mathcal{N}|X\prec\mathcal{N}$ be a countable elementary submodel with $X\cap H_\kappa\in S'$, so that $X\cap\omega_1\in T_0$.

Since $\calM$ is countable, we can pick a filter $G$ which is $\calM$-generic for $\P$ and contains $p$. Let \[\bar{q}=\kla{\vec{Q}^{\bar{q}},\vec{S}^{\bar{q}}}=\kla{\bigcup_{r\in G}\vec{Q}^r,\bigcup_{r\in G}\vec{S}^r}.\]
Using items                                                                                                                                                                                                                                                                                                                                                                                                                                                       \ref{item:GenericSequenceHasLengthOmega1-DSRP} and \ref{item:UnionOfGenericIsExhaustiveInFirstCoordinate} of Fact \ref{fact:PropertiesOfP-DSRP}, it follows that $\delta:=\dom(\vec{Q}^{\bar{q}})=\dom(\vec{S}^{\bar{q}})=X\cap\omega_1$, and that $\bigcup_{i<\delta}Q^{\bar{q}}_i=X\cap H_\kappa\in S$. Thus, if we set $q=\kla{\vQ^{\bar{q}}\verl(X\cap H_\kappa),\vS^{\bar{q}}}$, then $q\in\P$, and $q$ extends every condition in $G$. Moreover, since $D_n\in M$, for each $n<\omega$, it follows that $G$ meets each $D_n$, and hence that $p\ge q\in\Delta$, as desired.
\end{proof}

We are now ready to prove our characterization of the pairs that are $\SSP$-projective stationary.

\begin{thm}
\label{thm:SSP-AdequacyIsPSonVecT}
Let $\kappa>\omega_1$ be regular, $\mathcal{S}\sub\power([H_\kappa]^\omega)$ a nonempty collection of stationary subsets of $[H_\kappa]^\omega$, and $\vec{T}$ an $\omega_1$-sequence of pairwise disjoint stationary subsets of $\omega_1$. The following are equivalent:
\begin{enumerate}[label=(\arabic*)]
\item
\label{item:ProjectiveStationary}
$\mathcal{S}$ is projective stationary on $\vec{T}$.
\item
\label{item:Adequate}
$\kla{\mathcal{S},\vec{T}}$ is $\SSP$-projective stationary.
\end{enumerate}
\end{thm}

\begin{proof}
Let $D=\bigcup_{i<\omega_1}T_i$, let $t:D\To\omega_1$ be defined by $\alpha\in T_{t(\alpha)}$, and set $\P=\P^{\DSRP}_{\mathcal{S},\vec{T}}$.

\ref{item:ProjectiveStationary}$\implies$\ref{item:Adequate}:
Let $A\sub\omega_1$ be stationary, $p\in\P$ and $\dot{C}\in\V^\P$ such that $p\forces_\P$``$\dot{C}$ is a club subset of $\omega_1$.'' We will find a condition $q\le p$ in $\P$ that forces that $\dot{C}$ intersects $\check{A}$. Let $\theta$ be a sufficiently large regular cardinal, say $\theta>2^{2^\kappa}$.

\noindent\emph{Case 1:} There is an $i_0<\omega_1$ such that $A\cap T_{i_0}$ is stationary.

In this case, fix such an $i_0$. By assumption, for every $S\in\mathcal{S}$, $\{M\in S\st M\cap\omega_1\in A\cap T_{i_0}\}$ is stationary. By strengthening $p$ if necessary, we may assume that $i_0<\lambda^p$. Let $N\prec\kla{H_\theta,\in,p,\dot{C},\P,\mathcal{S},\vec{T},<^*}$ be a countable elementary submodel such that $M=N\cap H_\kappa\in S^p_{i_0}$ and $\delta=M\cap\omega_1\in A\cap T_{i_0}$. Let $g$ be $\P$-generic over $N$ with $p\in g$, and let $\vec{Q}=\bigcup_{q\in g}\vec{Q}^q$ and $\vec{S}=\bigcup_{q\in g}\vec{S}^q$.
Then $\vec{Q}$ is a sequence of length $\delta$, and $M=\bigcup_{i<\delta}Q_i\in S_{i_0}$. So since $\delta\in T_{i_0}$, $q=\kla{\vec{Q}\verl M,\vec{S}}$ is a condition that strengthens $p$ and forces that $\delta\in\dot{C}$, since $\dot{C}^g$ is club in $\delta$. Since $\delta\in A$, this means that $q$ forces that $\dot{C}$ intersects $\check{A}$, as desired.

\noindent\emph{Case 2:} $A\setminus D$ is stationary.

Let $N\prec\kla{H_\theta,\in,\mathcal{S},\vec{T},\P,p,\dot{C},<^*}$ be countable with $N\cap\omega_1=\delta\in A\setminus D$. Let $g\sub\P$ be $N$-generic with $p\in g$. Let $\vec{Q}=\bigcup_{q\in g}\vec{Q}^q$ and $\vec{S}=\bigcup_{q\in g}\vec{S}^q$. Since $\delta=\dom(\vQ)\notin D$, it follows that $q=\kla{\vec{Q}\verl(N\cap H_\kappa),\vS}\in\P$, and since $\dot{C}^g$ is club in $\delta$, it follows that $q$ forces that $\check{\delta}\in\dot{C}$, hence that $\check{A}\cap\dot{C}\neq\leer$.

\noindent\emph{Case 3:} Cases 1 and 2 fail.

Then $A\cap D$ is stationary and for all $i<\omega_1$, $A\cap T_i$ is nonstationary. Fix, for every $i<\omega_1$, a club $C_i\sub\omega_1$ disjoint from $A\cap T_i$. Let $A^*=A\cap D\cap(\diag_{i<\omega_1}C_i)$. Then $A^*$ is stationary and has the property that for all $\alpha\in A^*$ and all $\beta<\alpha$, $\alpha\notin T_\beta$. So $A^*\sub Z=\{\alpha\in D\st\forall\beta<\alpha\ \alpha\notin T_\beta\}$, and by
condition \ref{item:StrangeCondition} of Definition \ref{def:ProjectiveStationaryOnVecT}, $\bigcup\mathcal{S}$ is projective stationary on $Z$. Thus, we can pick a countable $N\prec\kla{H_\theta,\in,\P,p,\dot{C},A^*,\mathcal{S},\vec{T}}$ such that $\delta=N\cap\omega_1\in A^*$ and $M=N\cap H_\kappa\in\bigcup\mathcal{S}$. Let $S\in\mathcal{S}$ be such that $M\in S$. Let $g\sub\P$ be $N$-generic for $\P$. Let $\vec{Q}=\bigcup_{q\in g}\vec{Q}^q$ and $\vec{S}=\bigcup_{q\in g}\vec{S}^q$. Then $\dom(\vS)=\dom(\vQ)=\delta\in A^*$, and it follows that $t(\delta)\ge\delta$. Thus, $t(\delta)\notin\dom(f)$, and we can extend $\vS$ to some $\vS'$ of length $t(\delta)+1$ so that $\vS'\rest\delta=\vS$ and $S'_{t(\delta)}=S$. We can then let $\vec{Q}'=\vec{Q}\verl M$, resulting in a condition $q=\kla{\vec{Q}',\vS'}$ extending $p$ and forcing that $\delta\in\dot{C}\cap\check{A}^*$. Note that $A^*\sub A$.

Thus, in each case, we have found an extension $q$ of $p$ forcing that $\dot{C}$ intersects $\check{A}$. Thus, the stationarity of $A$ is preserved by $\P$, and since this holds for any stationary subset of $\omega_1$, $\P$ is stationary set preserving, that is, $\kla{\mathcal{S},\vec{T}}$ is $\SSP$-projective stationary.

\ref{item:Adequate}$\implies$\ref{item:ProjectiveStationary}:
Let $\P$ be stationary set preserving. We have to show that $\mathcal{S}$ is projective stationary on $\vec{T}$. This amounts to proving the two conditions listed in Definition \ref{def:ProjectiveStationaryOnVecT}.

For condition \ref{item:PSonEachTi}, let $i<\omega_1$, $S\in\mathcal{S}$, and let $A\sub T_i$ be stationary. We have to show that $S_A=\{M\in S\st M\cap\omega_1\in A\}$ is a stationary subset of $[H_\kappa]^\omega$. If not, then let $C\sub[H_\kappa]^\omega$ be club with $S_A\cap C=\leer$. Let $G$ be $\P$-generic over $\V$, such that $G$ contains a condition $p$ with $i<\lambda^p$ and $S^p_i=S$. Let $\vec{Q}=\bigcup_{q\in G}\vec{Q}^q$ and $\vec{S}=\bigcup_{q\in G}\vec{S}^q$. In $\V[G]$, $A$ is still stationary, so there is a
\[\delta\in A\cap\{\alpha<\omega_1\st Q_\alpha\cap\omega_1=\alpha\}\cap\{Q_\alpha\cap\omega_1\st Q_\alpha\in C\}.\]
But then $\delta=Q_\delta\cap\omega_1$ and $Q_\delta\in C$, and since $\delta\in A\sub T_i$, we have that $Q_\delta\in S_i$. So $Q_\delta\in S_A\cap C\neq\leer$. Since $C$ was arbitrary, this shows that $S_A$ is stationary, as claimed.

For condition \ref{item:StrangeCondition}, suppose $A\sub D$ is stationary in $\omega_1$ and has the property that for all $\alpha\in A$ and all $\beta<\alpha$, $\alpha\notin T_\beta$. Letting $S^*=\bigcup\mathcal{S}$, we have to show that
\[S^*_A=\{M\in S^*\st M\cap\omega_1\in A\}\]
is stationary. So let $C\sub[H_\kappa]^\omega$ be club. Let $G$ be $\P$-generic, and let $\vec{Q}=\bigcup_{q\in g}\vec{Q}^q$ and $\vec{S}=\bigcup_{q\in g}\vec{S}^q$. Let
\[\delta\in A\cap\{\alpha<\omega_1\st Q_\alpha\cap\omega_1=\alpha\}\cap\{Q_\alpha\cap\omega_1\st Q_\alpha\in C\}.\]
This is possible, because $A$ is stationary in $\V[G]$. It follows that $\delta=Q_\delta\cap\omega_1\in A\sub D$, so that $t(\delta)$ is defined and $Q_\delta\in S_{t(\delta)}\in\mathcal{S}$. It follows that $Q_\delta\in S^*_A\cap C$.
\end{proof}

\begin{remark}
If the nonstationary ideal on $\omega_1$ is $\omega_2$-saturated, then it was shown in \cite{FengJech:ProjectiveStationarityAndSRP} that for every stationary subset $S$ of $[H_\kappa]^\omega$, where $\kappa\ge\omega_2$ is regular, there is a stationary set $D\sub\omega_1$ such that $S$ is projective stationary on $D$. By the previous remark, if $\vec{T}$ is any partition of such a $D$ into stationary sets, and this partition is maximal, then $\mathcal{S}=\{T\sub[H_\kappa]^\omega\st T\ \text{is projective stationary on $D$}\}$ is projective stationary on $\vec{T}$, and $S\in\mathcal{S}$. 
\end{remark}

\section{The subcomplete fragment of the diagonal strong reflection principle}
\label{sec:SC-DSRP}

We will now carry out the analysis of Section \ref{sec:SSP-DSRP} for the class of $\infty$-subcomplete forcing, that is, $\Gamma=\infSC$. Thus, we have to find a description of the pairs $\kla{\mathcal{S},\vec{T}}$ that are $\infSC$-projective stationary. To this end, we first make the following definition, corresponding to the notion of projective stationarity on a subset of $\omega_1$.

\begin{defn}
\label{def:SpreadOutOnA}
Let $D$ be a set, usually of the form $H_\kappa$, for some regular $\kappa>\omega_1$, and let $T\sub\omega_1$.
Then a set $S\sub[D]^\omega$ with $\bigcup S=D$ is \emph{spread out on $T$} if for all sufficiently large $\theta$, whenever $\tau$, $A$, $X$ and $a$ are such that $H_\theta\sub L_\tau^A=N\models\ZFCm$, $S,a,T,\theta\in X\prec N$, $X$ is countable and full, \emph{and $X\cap\omega_1\in T$}, then there are a $Y\prec N$ and an isomorphism
$\pi:N|X\To N|Y$ such that $\pi(a)=a$ and $Y\cap D\in S$.
\end{defn}

As with projective stationarity on a nonstationary set, this notion is also vacuous in this case, see Remark \ref{rem:VacuousUnlessStationary}.

\begin{remark}
\label{rem:SpreadOutVacuousUnlessStationary}
Let $\kappa>\omega_1$ be regular, and let $S\sub[D]^\omega$ be stationary in $D$. If $T\sub\omega_1$ is nonstationary, then $S$ is spread out on $T$.
\end{remark}

\begin{proof}
Let $\theta>\kappa$, and let $\tau$, $A$, $X$ and $a$ be as in Definition \ref{def:SpreadOutOnA}. In particular, $T\in X$. By elementarity, $X$ sees that $T$ is not stationary, so there is a club set $C\sub\omega_1$ in $X$, disjoint from $T$. Letting $\delta=X\cap\omega_1$, it follows that $C\cap\delta$ is unbounded in $\delta$, and hence that $\delta\in C$. Thus, $X\cap\omega_1\notin T$, and hence, the condition stated in Definition \ref{def:SpreadOutOnA} holds vacuously.
\end{proof}

The following definition, which corresponds to Definition \ref{def:ProjectiveStationaryOnVecT} in the stationary set preserving case, is designed to capture $\infSC$-projective stationarity.

\begin{defn}
\label{def:SpreadOutOnVecT}
Let $\kappa>\omega_1$ be regular, $\mathcal{S}$ a nonempty collection of stationary subsets of $[H_\kappa]^\omega$, and $\vec{T}$ a sequence of pairwise disjoint stationary subsets of $\omega_1$. Then $\mathcal{S}$ is \emph{spread out on $\vec{T}$} if
\begin{enumerate}[label=(\alph*)]
\item
\label{item:EverybodySpreadOutOnEachPiece}
for every $i<\omega_1$ and for every $S\in\mathcal{S}$, $S$ is spread out on $T_i$.
\item
\label{item:UnionSpreadOutOnLeftovers}
$\bigcup\mathcal{S}$ is spread out on $\bigcup_{i<\omega_1}T_i\ohne\bigtriangledown_{i<\omega_1}T_i$.
\end{enumerate}
\end{defn}

As before, condition \ref{item:UnionSpreadOutOnLeftovers} is vacuous if $\bigcup_{i<\omega_1}T_i\ohne\bigtriangledown_{i<\omega_1}T_i$ is nonstationary, that is, if $\vT$ is maximal, and as before, maximality results in a considerable simplification of the concept.

\begin{remark}
\label{rem:SpreadOutOnvecTSimplified}
Let $\kappa>\omega_1$ be regular, let $\vT=\seq{T_i}{i<\omega_1}$ be a sequence of pairwise disjoint stationary subsets of $\omega_1$ that is maximal, and let $\mathcal{S}$ be a collection of subsets of $[H_\kappa]^\omega$. Then $\mathcal{S}$ is spread out on $\vT$ iff every $S\in\mathcal{S}$ is spread out on $\bigcup_{i<\omega_1}T_i$.

Thus, if $\vT$ is a maximal partition of $\omega_1$ into stationary sets, then $\mathcal{S}$ is spread out on $\vT$ iff every $S\in\mathcal{S}$ is spread out.
\end{remark}

\begin{proof}
Set $D=\bigcup_{i<\omega_1}T_i$.
For the implication from left to right, fix $S\in\mathcal{S}$. let $\theta$ be sufficiently large, and let $H_\theta\sub L_\tau^A=N\models\ZFCm$. Let $X$ be countable and full, with $N|X\prec N$, and assume that $\theta, S, D\in X$. Fix some $a\in X$. By a version of Fact \ref{fact:WeaklySpreadOutImpliesSpreadOut} may also assume that $\vT\in X$. But then, $Z=D\ohne\bigtriangledown_{i<\omega_1}T_i$ is also in $X$, and $Z$ is nonstationary, by assumption. As in the proof of \ref{rem:SpreadOutVacuousUnlessStationary}, it follows that $\delta=X\cap\omega_1\notin Z$. Now suppose that $\delta\in D$. Since $\delta\notin Z$, this means that $\delta\in T_{i_0}$, for some $i_0<\delta$. But since $S$ is spread out on $T_{i_0}$, there are $\pi$, $Y$ such that $\pi:N|X\To N|Y\prec N$ is an isomorphism that fixes $a$, and such that $Y\cap H_\kappa\in S$, as wished.

The converse is trivial, because if every $S\in\mathcal{S}$ is spread out on $\bigcup_{i<\omega_1}T_i$, then it is trivially also spread out on each $T_i$ (we may assume that $D$ belongs to the relevant $X$). And condition \ref{item:UnionSpreadOutOnLeftovers} of Definition \ref{def:SpreadOutOnVecT} is vacuous, since $\vT$ is maximal.
\end{proof}

As before, the reader may focus on the situation where $\vT$ is maximal, but we treat the general case here.

\begin{thm}
\label{thm:CharacterizationOfInfSCAdquacy}
Let $\kappa>\omega_1$ be regular, $\vec{T}$ an $\omega_1$-sequence of pairwise disjoint stationary subsets of $\omega_1$, and $\mathcal{S}$ a nonempty collection of stationary subsets of $[H_\kappa]^\omega$. The following are equivalent.
\begin{enumerate}[label=(\arabic*)]
\item
\label{item:SpreadOutOnUnionOfVecT}
$\mathcal{S}$ is spread out on $\vec{T}$. 
\item
\label{item:InfSCAdequate}
$\kla{\mathcal{S},\vec{T}}$ is $\infSC$-projective stationary.
\end{enumerate}
\end{thm}

\begin{proof} Let $D=\bigcup_{i<\omega_1}T_i$, let $t:D\To\omega_1$ be defined by $\alpha\in T_{t(\alpha)}$, and let $\P=\P^{\DSRP}_{\mathcal{S},\vec{T}}$. We treat each implication separately.

\ref{item:SpreadOutOnUnionOfVecT}$\implies$\ref{item:InfSCAdequate}:
Assuming that $\mathcal{S}$ is spread out on $T_i$, for every $i<\omega_1$, we have to show that $\P$ is $\infty$-subcomplete. To this end, let $\theta$ be large enough for Definition \ref{def:SpreadOutOnA} to apply to every $T_i$, every $S\in\mathcal{S}$, as well as to $\bigcup\mathcal{S}$ and $D\ohne\bigtriangleup_{i<\omega_1}T_i$. Let $N=L_\tau^A\models\ZFCm$ with $H_\theta\sub N$, and let $\P\in X\prec N$ be countable and full. Let $a$ be some member of $X$, and let $\sigma:\bN\To X$ be the inverse of the Mostowski collapse of $X$, $\bN$ transitive. Let $\bP=\sigma^{-1}(\P_S)$, $\bar{a}=\sigma^{-1}(a)$, and let $\bG\sub\bP$ be $\bN$-generic. 
As usual, we may assume that certain parameters are in $X$; see \cite[p.~116, Lemma 2.5]{Jensen2014:SubcompleteAndLForcingSingapore}. Here, we will assume that $\mathcal{S},\vec{T}\in X$.
Let $\bkappa=\sigma^{-1}(\kappa)$, $\bar{\mathcal{S}}=\sigma^{-1}(\mathcal{S})$ and $\vec{\bar{T}}=\sigma^{-1}(\vec{T})$. It follows from Lemma \ref{lem:PropertiesOfTheGeneric} that if we set
$\vec{\bar{Q}}=\bigcup\{\vQ^{\bar{p}}\st\bar{p}\in\bar{G}\}$ and $\vec{\bS}=\bigcup\{\vS^{\bar{p}}\st\bar{p}\in\bar{G}\}$, then $\vec{\bar{Q}}$ and $\vec{\bS}$ are sequences of length $\omega_1^\bN$.

Let $\delta=X\cap\omega_1=\omega_1^\bN$.

\noindent\emph{Case 1:} $\delta\notin D$.

In this case we define
\[q=\big\langle\seq{\sigma(\bar{Q}_i)}{i<\delta}\verl(X\cap H_\kappa)\ ,\ \seq{\sigma(\bar{S}_i)}{i<\delta}\big\rangle.\]
Then $q\in\P$, since $X\cap\omega_1\notin D$. Moreover, $q$ extends every member of $\sigma``\bar{G}$. Thus, $q$ forces that $\sigma$ itself satisfies the subcompleteness conditions \ref{item:FirstSubcompletenessCondition}-\ref{item:SupremumCondition} of Definition \ref{def:(ininifty-)subcompleteness}.

\noindent\emph{Case 2:} $\delta\in D$.

Let $i_0<\omega_1$ be such that $\delta\in T_{i_0}$, that is, $i_0=t(\delta)$.

\noindent\emph{Case 2.1:} $i_0<\delta$.

In this case, let $\bar{S}=\bS_{i_0}$. $\bar{S}$ is then in $\bar{\mathcal{S}}$, and so, $S=\sigma(\bar{S})\in\mathcal{S}\cap X$. In particular, $S$ is spread out on $T_{i_0}$. Moreover, $T_{i_0}\in X$.
So, since $\delta=X\cap\omega_1\in T_{i_0}$, we can choose a $Y\prec N$ with $Y\cap H_\kappa\in S$ and an isomorphism $\pi:N|X\To N|Y$ that fixes $a$, $S$ and $\P$. Let $\sigma'=\pi\circ\sigma:\bN\prec N$.
Let
\[q=\big\langle\seq{\sigma'(\bar{Q}_i)}{i<\delta}\verl(Y\cap H_\kappa)\ ,\ \seq{\sigma'(\bS_i)}{i<\delta}\big\rangle.\]
Since $Y\cap H_\kappa\in S$, it follows that $q\in\P$ (note that $X\cap\omega_1=Y\cap\omega_1=\delta\in T_{i_0}$, and $Y\cap H_\kappa\in S=\sigma(\bS_{i_0})=\sigma'(\bS_{i_0})$), and whenever $G\ni q$ is $\P_S$-generic over $\V$, then $\sigma'``\bG\sub G$. Since $\sigma'(\bar{a})=a$, the conditions defining $\infty$-subcompleteness are satisfied.

\noindent\emph{Case 2.2:} $i_0\ge\delta$.

In this case, $\delta\in Z=D\ohne\bigtriangledown_{i<\omega_1}T_i$. By assumption, $\bigcup\mathcal{S}$ is spread out on $Z$. Moreover, $\bigcup{\mathcal{S}}$ and $Z$ are in $X$. Let $Y$, $\sigma'$ be such that $\sigma':N|X\To N|Y$ is an isomorphism fixing $a$, $\mathsf{S}$, $\vT$ and $\P$, and such that $Y\cap H_\kappa\in\bigcup\mathcal{S}$. Let $S\in\mathcal{S}$ be such that $Y\cap H_\kappa\in S$. Let $\vS'$ be a sequence of length $i_0+1$ extending $\seq{\sigma'(\bS_i)}{i<\delta}$ with $S'_{i_0}=S$ and $S'_i\in\mathsf{S}$ for every $i\le i_0$. Let
\[q=\big\langle\seq{\sigma'(\bar{Q}_i)}{i<\delta}\verl(Y\cap H_\kappa),\vS'\big\rangle.\]
Then $q$ is a condition, forcing that $\sigma'``\bar{G}\sub\dot{G}$.

\ref{item:InfSCAdequate}$\implies$\ref{item:SpreadOutOnUnionOfVecT}:
To prove that condition \ref{item:EverybodySpreadOutOnEachPiece} of Definition \ref{def:SpreadOutOnVecT} holds, fix an $S\in\mathcal{S}$ and an $i_0<\omega_1$.
Let $\theta$ be large enough to verify that $\P$ is $\infty$-subcomplete.
Let $\tau,a,X,A,N$ be as in Definition \ref{def:SpreadOutOnA}. So $X\prec N=L_\tau^A$ is countable and full, $S,a,T_{i_0}\in X$, and suppose that $\delta=X\cap\omega_1\in T_{i_0}$, that is, $t(\delta)=i_0$. By (a variation of) Fact \ref{fact:WeaklySpreadOutImpliesSpreadOut}, we may assume that $X$ contains certain parameters we care about. So let us assume that $\mathcal{S},\vec{T}\in X$. Since $T_{i_0}$ and $\vec{T}$ are in $X$, it follows that $i_0\in X$, and hence that $i_0<\delta$.
Let $\sigma:\bN\To X$ be the transitive isomorph of $X$, and let $\bar{\P}$, $\bar{S}$, $\bar{\mathcal{S}}$, $\vec{\bar{T}}$ be the preimages of $\P$, $S$, $\mathcal{S}$, $\vec{T}$ under $\sigma$, respectively.

Let $\bar{G}$ be $\bar{\P}$-generic over $\bN$, containing a condition $\bar{p}$ such that $S^{\bar{p}}_{i_0}=\bar{S}$. By assumption, there is a condition $q\in\P$ such that whenever $G$ is $\P$-generic over $\V$ and contains $q$, then there is in $\V[G]$ an elementary embedding $\sigma':\bN\To N$ with $\sigma'(\bar{a})=a$, $\sigma'(\bar{S})=S$, $\sigma'(\bar{\mathcal{S}})=\mathcal{S}$, $\sigma'(\vec{\bar{T}})=\vT$, and such that $\sigma'``\bar{G}\sub G$. Note that for any $S'\in\mathcal{S}$, and any $i<\omega_1$, $S'$ is projective stationary on $T_i$, since $S'$ is even spread out on $T_i$. This can be easily shown directly.
Hence, by Lemma \ref{lem:CountableDistributivityOfP-DSRP}, $\P$ is countably distributive, so that $\sigma'$ already exists in $\V$. We have already argued that the union of the first coordinates of conditions in $\bar{G}$ is of the form $\seq{\bar{Q}_i}{i<\omega_1^\bN}$, where $\bigcup_{i<\delta}\bar{Q}_i=H_\bkappa^\bN$, and that the union of the second coordinates is a sequence $\seq{\bar{S}_i}{i<\delta}$. Now let $r\in G$ be a condition with $\delta^r\ge\delta$, and let $Y=\ran(\sigma')$. Then
\[Q^r_\delta=\bigcup_{i<\delta}Q^r_i=\bigcup_{i<\delta}\sigma'``\bar{Q}_i=\sigma'``H_\bkappa^\bN=Y\cap H_\kappa.\]
Moreover, $Y\cap\omega_1=X\cap\omega_1\in T_{i_0}$, so that $Q^r_\delta\in S^r_{i_0}=\sigma'(\bar{Q}_{i_0})=S$, by clause (3) of Definition \ref{def:ForcingDSRP}. Thus, $Y\cap H_\kappa\in S$, and letting $\pi=\sigma'\circ\sigma^{-1}$, we have that $\pi:N|X\prec N|Y$ is an isomorphism fixing $a$, showing that $S$ is spread out on $T_{i_0}$.

To prove condition \ref{item:UnionSpreadOutOnLeftovers} of Definition \ref{def:SpreadOutOnVecT}, we start in the same setup, but we assume that $\delta\in\bigcup_{i<\omega_1}T_i\ohne\bigtriangledown_{i<\omega_1}T_i$, that is, $\delta\in T_{i_0}$, where $i_0\ge\delta$. Let $\bG$ be generic over $\bN$ for $\bar{\P}$, and let $q\in\P$ force the existence of a $\sigma':\bN\prec N$ as before, moving $\vec{\bT}$, $\bar{\P}$, $\bar{\mathcal{S}}$ and $\bar{a}$ the same way as $\sigma$, and so that if $G$ is $\P$-generic with $q\in G$, then $\sigma'``\bG\sub G$. As before, it follows that $\sigma'\in\V$. Let $r\in G$ be such that $\delta^r\ge\delta$. It follows that, letting $Y=\ran(\sigma')$, $Q^r_\delta=Y\cap H_\kappa$. So, since $\delta\in T_{i_0}$, $i_0<\lambda^r$ and $Q^r_\delta\in S^r_{i_0}\in\mathcal{S}$. Hence, $Y\cap H_\kappa\in\bigcup\mathcal{S}$, and $\pi=\sigma'\circ\sigma^{-1}$ can serve as our wanted isomorphism.
\end{proof}

\section{Consequences of the $\Gamma$-fragment of \DSRP}
\label{subsec:DSRPconsequences}

Now that we have characterizations of the pairs $\kla{\mathcal{S},\vT}$ that are $\Gamma$-projective stationary, if $\Gamma$ is either the class of stationary set preserving or subcomplete forcing notions, we should like to describe some consequences of the corresponding $\Gamma$-fragment of $\DSRP$. First, let us summarize the most important consequence of what was done in Sections \ref{sec:SSP-DSRP} and \ref{sec:SC-DSRP}.

\begin{thm}
\label{thm:Gamma-DSRPimpliesExactDSRP}
Let $\Gamma$ be either the class of stationary set preserving or of $\inf$-subcomplete forcing notions. Let $\kappa>\omega_1$ be regular, and suppose that $\Gamma$-$\DSRP(\kappa)$ holds. Then:
\begin{enumerate}[label=(\alph*)]
\item
\label{item:SummaryTheoremPartA}
If $\mathcal{S}\neq\leer$ is such that every $S\in\mathcal{S}$ is $\Gamma$-projective stationary in $H_\kappa$, then $\eDSRP(\mathcal{S})$ holds.	
\item
\label{item:SummaryTheoremPartB}
If $A\sub\omega_1$ is a stationary set such that every $S\in\mathcal{S}\neq\leer$ is $\Gamma$-projective stationary in $H_\kappa$ on $A$, then $\DSRP(\mathcal{S})$ holds.	
\end{enumerate}
\end{thm}

\begin{proof} Part \ref{item:SummaryTheoremPartA}:
if $\Gamma$ is the class of stationary set preserving forcing notions,  then $\Gamma$-projective stationarity is just the usual concept of projective stationarity. So let $\mathcal{S}$ be a nonempty collection of projective stationary sets in $H_\kappa$. Let $\vT$ be a maximal partition of $\omega_1$ into stationary sets. By Remark \ref{rem:ProjStatOnVecTifTisMaximal}, $\kla{\mathcal{S},\vT}$ is \SSP-projective stationary, so by assumption, $\DSRP(\mathcal{S},\vT)$ holds. But since $\vT$ is a partition of all of $\omega_1$, this implies $\eDSRP(\mathcal{S})$, by Remark \ref{remark:GettingExactDSRPwhenTisTotal}.

The case where $\Gamma$ is the class of all $\inf$-subcomplete forcing notions is handled similarly. This time, $\Gamma$-projective stationarity means being spread out. Given $\mathcal{S}$ and a partition $\vT$ as above, it follows by Remark \ref{rem:SpreadOutOnvecTSimplified} that $\kla{\mathcal{S},\vT}$ is $\infSC$-projective stationary, so that $\DSRP(\mathcal{S},\vT)$ holds, which again implies $\eDSRP(\mathcal{S})$, as $\vT$ is a partition of $\omega_1$.

Part \ref{item:SummaryTheoremPartB} is similar. We can work with a maximal partition $\vT$ of $A$ into stationary sets now.
\end{proof}

The relationship between the diagonal strong reflection principle and other diagonal reflection principles is maybe best understood in an analogy to the relationship between the strong reflection principle and other reflection principles. In fact, it may be easiest to understand the difference by thinking about Friedman's problem and the reflection principle, in the context of reflection of stationary sets of ordinals. Friedman's problem at an uncountable regular cardinal $\kappa$ greater than $\omega_1$ says that whenever $A\sub S^\kappa_\omega$ is stationary, then there is a closed subset $C$ of $A$ of order type $\omega_1$. Letting $\rho=\sup C$, then, $A\cap\rho$ is stationary, that is, $A$ \emph{reflects} at $\rho$. But $A\cap\rho$ is not only stationary; it contains a club. In preparation for the following subsections, let us define some concepts that capture the difference between reflection in the usual sense and the kind of reflection resulting from strong reflection principles. The terminology around exact reflection comes from \cite[Def.~3.13]{Fuchs:CanonicalFragmentsOfSRP}.

\begin{defn}
Let $\kappa$ be an ordinal of uncountable cofinality, and let $A\sub \kappa$ be stationary in $\kappa$. An ordinal $\rho<\kappa$ of uncountable cofinality is a \emph{reflection point} of $A$ if $A\cap\rho$ is stationary in $\rho$. It is an \emph{exact reflection point} of $A$ if $A\cap\rho$ contains a club in $\rho$. Given a regular cardinal $\delta$, the $\delta$-\emph{trace} of $A$, $\Tr_\delta(A)$, is the set of all reflection points of $A$ that have cofinality $\delta$, and the \emph{exact $\delta$-trace} of $A$, $\eTr_\delta(A)$, is the set of all exact reflection points of $A$ that have cofinality $\delta$.

If $\mathcal{S}$ is a collection of stationary subsets of $\kappa$, then $\rho$ is a \emph{simultaneous reflection point} of $\mathcal{S}$ if $\rho$ is a reflection point of every $A\in\mathcal{S}$. It is an \emph{exact simultaneous reflection point} of $\mathcal{S}$ if it is a simultaneous reflection point of $\mathcal{S}$ and $(\bigcup\mathcal{S})\cap\rho$ contains a club in $\rho$. Again fixing a regular cardinal $\delta$, the \emph{$\delta$-trace} of $\mathcal{S}$, $\Tr_\delta(\mathcal{S})$, is the set of all simultaneous reflection points of $\mathcal{S}$ that have cofinality $\delta$, and the \emph{exact $\delta$-trace} of $\mathcal{S}$, $\eTr_\delta(\mathcal{S})$, is the set of all exact simultaneous reflection points of $\mathcal{S}$ that have cofinality $\delta$.

	Since we will be mainly interested in the case that $\delta=\omega_1$, we will drop mention of $\delta$ if $\delta=\omega_1$, that is, $\eTr(\mathcal{S})$ means $\eTr_{\omega_1}(\mathcal{S})$. 
\end{defn}

Thus, Friedman's problem at $\kappa$ says that every stationary subset of $S^\kappa_\omega$ has an exact reflection point. In fact, let us define the exact versions of some classical reflection principles for stationary sets of ordinals.

\begin{defn}
Suppose $\lambda$ is a cardinal of uncountable cofinality. Let $A\sub \lambda$. Let $\kappa$ be a cardinal, and let $\delta$ be a regular cardinal. Then $\Refl_\delta({<}\kappa,A)$ says that whenever $\mathcal{S}$ is a collection of stationary subsets of $A$ that has cardinality less than $\kappa$, then $\Tr_\delta(\mathcal{S})\neq\leer$. We write $\Refl_\delta(\kappa,A)$ for $\Refl_\delta({<}\kappa^+,A)$.

Similarly, $\eRefl_\delta({<}\kappa,A)$ says that whenever $\mathcal{S}$ is a collection of stationary subsets of $A$ that has size less than $\kappa$, then $\eTr_\delta(\mathcal{S})\neq\leer$. As before, we write $\eRefl_\delta(\kappa,A)$ for $\eRefl_\delta({<}\kappa^+,A)$.	

And as before, we will drop mention of $\delta$ if $\delta=\omega_1$, so that $\eRefl(\kappa,A)$ means $\eRefl_{\omega_1}(\kappa,A)$.
\end{defn}

We are here most concerned with the principles of the form $\Refl(\omega_1,A)$ and $\eRefl(\omega_1,A)$. It is shown in \cite{Fuchs:CanonicalFragmentsOfSRP} that $\eRefl(\omega_1,A)$ is equivalent to a simultaneous version of Friedman's Problem that has its origins in \cite{FMS:MM1}:

\begin{obs}[{\cite[Obs.~3.14]{Fuchs:CanonicalFragmentsOfSRP}}]
Let $\kappa>\omega_1$ be regular and fix a stationary subset $A$ of $\kappa$. The following are equivalent:
\begin{enumerate}
\item Whenever $\mathcal{S}=\{A_i\st i<\omega_1\}$ is a set of stationary subsets of $A$, there is a partition $\seq{T_i}{i<\omega_1}$ of $\omega_1$ into stationary sets and a normal function $f:\omega_1\To\kappa$ such that for every $i<\omega_1$, $f``T_i\sub A_i$.
\item $\eRefl(\omega_1,A)$ holds.
\item For any set $\mathcal{S}$ of stationary subsets of $A$ that has size $\omega_1$, $\eTr(\mathcal{S})$ is stationary in $\kappa$.	
\end{enumerate}
\end{obs}

Exact simultaneous reflection has consequences on cardinal arithmetic (and this was known since \cite{FMS:MM1}, even though this was not filtered through the simultaneous exact reflection principle):

\begin{fact}[{\cite[Fact 3.15]{Fuchs:CanonicalFragmentsOfSRP}}]
\label{fact:ExactReflectionAndCardinalArithmetic}
Let $\kappa>\omega_1$ be regular, and suppose there is a stationary $A\sub\kappa$ such that $\eRefl(\omega_1,A)$ holds. Then $\kappa^{\omega_1}=\kappa$.
\end{fact}

It was shown in \cite{FMS:MM1} that $\MM$ implies $\eRefl(\omega_1,S^{\kappa}_\omega)$, for any regular $\kappa>\omega_1$. \Todorcevic showed that already \SRP has this consequence, and in \cite{Fuchs:CanonicalFragmentsOfSRP}, it was shown that the $\infty$-subcomplete fragment of $\SRP$ implies this for $\kappa>2^\omega$.

The strong diagonal reflection principle is a principle of reflection of generalized stationarity, designed to capture exact versions of diagonal reflection. Note that an exact reflection point of some collection of stationary sets is a reflection point of each of those sets, but it is explicitly \emph{not} a reflection point of the complement of the union of these stationary sets. Thus, principles of exact reflection provide \emph{selective reflection:} points at which some sets reflect but others don't. The diagonal reflection principles, introduced by the first author, talk about reflection of generalized stationarity, and in a sense, they try to maximize the collection of sets that reflect. They are thus not designed to produce phenomena of exact reflection. For example, they do not imply $\eRefl(\omega_1,A)$, for any set $A$ stationary in $\omega_2$, since they don't imply that $2^{\omega_1}=\omega_2$, as we will show in Section \ref{sec:DRPlimitations} (compare with Fact \ref{fact:ExactReflectionAndCardinalArithmetic}).

We will present in the following two subsections some consequences of fragments of the diagonal strong reflection principle. First, we will focus on consequences that don't have much to do with the exact reflection \DSRP provides. These filter through certain versions of the diagonal reflection principle that were introduced in \cite{Cox:DRP}. In the subsection after that, we will provide some applications that do make use of the exact quality of the reflection \DSRP provides. These don't follow from the principles of \cite{Cox:DRP}.

\subsection{Consequences that filter through weak diagonal reflection principles}
\label{subsec:FilteredConsequences}

Let us begin by showing that \DSRP implies various ``weak'' diagonal reflection principles of \cite{Cox:DRP}, as well as some slight modifications thereof. For the present purposes, we say that a set $N$ is \emph{internally approachable} if it is the union of an $\in$-chain $\seq{N_\alpha}{\alpha<\omega_1}$ such that for every $\alpha<\omega_1$, $\seq{N_\xi}{\xi<\alpha}\in N$.

\begin{lem}
\label{lem:DSRPimpliesDRP}
Let $\kappa$ be regular, and let $\mathcal{S}\sub\power([H_\kappa]^\omega)$ be a collection of stationary sets such that $\DSRP(\mathcal{S})$ holds. Then
\begin{enumerate}[label=(\arabic*)]
\item
\label{item:DSRPimpliesDRPIA}
The principle $\mathsf{w}\DRP_{\IA}(\mathcal{S})$ holds: whenever $\theta$ is large enough that $\mathcal{S}\sub H_\theta$, there are stationarily many $W\in[H_\theta]^{\omega_1}$ such that:
\begin{enumerate}
\item $W\cap H_\kappa$ is internally approachable,
\item for every $S\in W\cap\mathcal{S}$, $S\cap[W\cap H_\kappa]^\omega$ is stationary in $W\cap H_\kappa$.	
\end{enumerate}
\item
\label{item:DSRPimpliesDRPIA+}
The following slight strengthening of $\mathsf{w}\DRP_{\IA}(\mathcal{S})$ holds: whenever $\theta$ is large enough that $\mathcal{S}\sub H_\theta$, there are stationarily many $W\in[H_\theta]^{\omega_1}$ such that for every $S\in W\cap\mathcal{S}$ and every regular $\bkappa\in W\cap[\omega_2,\kappa]$, $W\cap H_\bkappa$ is internally approachable and $(S\projectdown H_\bkappa)\cap[W\cap H_\bkappa]^\omega$ is stationary in $W\cap H_\bkappa$.
\end{enumerate}
\end{lem}

\begin{proof}
For \ref{item:DSRPimpliesDRPIA}, let $\theta$ be regular and large enough that $\mathcal{S}\sub H_\theta$. We know by $\DSRP(\mathcal{S})$ that there are stationarily many $W\in[H_\theta]^{\omega_1}$ such that $\omega_1\sub W$ and there is a diagonal chain $\vec{Q}$ through $\mathcal{S}$ up to $W$. We claim that each such $W$ belongs to the set defined in \ref{item:DSRPimpliesDRPIA}.
Note that $\vec{Q}$ witnesses that $W\cap H_\kappa$ is internally approachable. Now let $S\in W\cap\mathcal{S}$. We have to show that $S\cap[W\cap H_\kappa]^\omega$ is stationary in $W\cap H_\kappa$. Let
\[T=\{\alpha<\omega_1\st Q_\alpha\in S\}.\]
Since $\vQ$ is a diagonal chain through $\mathcal{S}$ up to $W$, $T$ is stationary. Now let $f:[W\cap H_\kappa]^{{<}\omega}\To W\cap H_\kappa$. We have to find an $x\in S\cap[W\cap H_\kappa]^\omega$ that's closed under $f$. Clearly, the set of $\alpha<\omega_1$ such that $f``[Q_\alpha]^{{<}\omega}\sub Q_\alpha$ is club in $\omega_1$. Hence, there is such an $\alpha$ in $T$. But then, $x=Q_\alpha\in S$ is as wished.


For \ref{item:DSRPimpliesDRPIA+}, we argue mostly as above. Given a $W$ as above, let $\vQ$ be a diagonal chain through $\mathcal{S}$ up to $W$, and let $\bkappa\in[\omega_2,\kappa]\cap W$. Then the sequence $\seq{Q_\alpha\cap H_\bkappa}{\alpha<\omega_1}$ witnesses that $W\cap H_\bkappa$ is internally approachable. Letting $S\in\mathcal{S}\cap W$, and letting $T$ be the stationary set of countable $\alpha$ such that $Q_\alpha\in S$, we have that for all $\alpha\in T$, $Q_\alpha\cap H_\bkappa\in S\projectdown H_\bkappa$. As above, given $f:[W\cap H_\bkappa]^{{<}\omega}\To W\cap H_\bkappa$, we can now find an $\alpha\in T$ such that $Q_\alpha\cap H_\bkappa$ is closed under $f$, and $Q_\alpha\cap H_\bkappa$ is then in $S\projectdown H_\bkappa$.
\end{proof}

\begin{remark}
\label{rem:Monotonicity}
In the notation of the previous lemma, if $\mathcal{T}\sub\mathcal{S}$, then $\mathsf{w}\DRP_\IA(\mathcal{S})\implies\mathsf{w}\DRP_\IA(\mathcal{T})$.
\end{remark}

This remark drives a point home that was made earlier: one cannot expect to get any phenomena of exact reflection from these principles. Yet they will, by design, imply certain diagonal reflection principles for sequences of stationary sets of ordinals. 

\begin{defn}
\label{def:SomeStationarySets}

The following collections of stationary sets will be focal for our analysis, for a regular cardinal $\kappa>\omega_1$:
\[
\mathcal{S}_{\lifting}(\theta)=\{\lifting(A,[H_\kappa]^\omega)\cap C\st A\sub S^\kappa_\omega\ \text{is stationary in $\kappa$ and
$C\sub[H_\kappa]^\omega$ is club}\}.\]
Further, for a forcing class $\Gamma$, let $\mathcal{S}_\Gamma(\kappa)$ be the collection of all $S\sub[H_\kappa]^\omega$ that are $\Gamma$-projective stationary in $H_\kappa$. 
\end{defn}

In \cite{Cox:DRP}, $\wDRPIA(\kappa)$ was defined as $\wDRPIA(\mathcal{S}_\SSP(\kappa))$, and $\wDRPIA$ states that $\wDRPIA(\kappa)$ holds for every regular $\kappa\ge\omega_2$. Thus, by Lemma \ref{lem:FA(Gamma)impliesGamma-DSRP}, Theorem \ref{thm:SSP-AdequacyIsPSonVecT} and Lemma \ref{lem:DSRPimpliesDRP}, we have the following implications:
\[\MM\implies\SSP\text{-}\DSRP(\kappa)\implies\wDRPIA(\kappa)\]
for every regular $\kappa>\omega_1$. If we similarly define $\infSC$-$\wDRPIA(\kappa)$ to be the principle $\wDRPIA(\mathcal{S}_{\infSC}(\kappa)$, then we obtain the corresponding implications
\[\infSCFA\implies\infSC\text{-}\DSRP(\kappa)\implies\infSC\text{-}\wDRPIA(\kappa)\]
for any regular $\kappa>\omega_1$, using Lemma \ref{lem:FA(Gamma)impliesGamma-DSRP}, Theorem \ref{thm:CharacterizationOfInfSCAdquacy} and Lemma \ref{lem:DSRPimpliesDRP}.



We now aim to find a connection to diagonal reflection principles of stationary sets of ordinals. Combining Theorem \ref{thm:DSRP=>eDSR}, Lemma \ref{lem:LiftingsAreSpreadOut} and Lemma \ref{lem:DSRPimpliesDRP}, we obtain:

\begin{cor}
\label{cor:DSRPimpliesDRPliftings}
Suppose $\kappa$ is a regular cardinal greater than $2^\omega$, and $\infSC$-$\DSRP(\kappa)$ holds. Then
$\mathsf{w}\DRP_\IA(\mathcal{S}_\lifting(\kappa))$ holds.

The same conclusion holds if $\kappa$ is a regular cardinal greater than $\omega_1$ and $\SSP$-$\DSRP(\kappa)$ holds.
\end{cor}

The principles of reflection of stationary sets of ordinals we are interested in here are of the following form.

\begin{defn}[{see \cite{Fuchs:DiagonalReflection}, \cite{Fuchs-LambieHanson:SeparatingDSR} and \cite{Larson:SeparatingSRP}}]
\label{defn:DiagonalReflection}
Let $\lambda$ be a regular cardinal, let $S\subseteq\lambda$ be stationary, and let $\kappa<\lambda$. The \emph{diagonal stationary reflection principle} $\DSR({<}\kappa,S)$ says that whenever $\langle S_{\alpha,i}\st\alpha<\lambda,i<j_\alpha\rangle$ is a sequence of stationary subsets of $S$, where $j_\alpha<\kappa$ for every $\alpha<\lambda$, then there are an ordinal $\gamma<\lambda$ of uncountable cofinality and a \emph{club} $F\subseteq\gamma$ such that for every $\alpha\in F$ and every $i<j_\alpha$, $S_{\alpha,i}\cap\gamma$ is stationary in $\gamma$.

The version of the principle in which $j_\alpha\le\kappa$ is denoted $\DSR(\kappa,S)$.

We will denote the collection of all ordinals less than some given ordinal $\lambda$ that have cofinality $\kappa$, for some regular cardinal $\kappa$, by $S^\lambda_\kappa$. Usually, in the present context, the set $S$ above will be of the form $S^\theta_\omega$, for some regular $\theta>\omega_1$.

If $F$ is only required to be \emph{unbounded,} then the resulting principle is called $\uDSR({<}\kappa,S)$, and if it is required to be \emph{stationary,} then it is denoted $\sDSR({<}\kappa,S)$.
\end{defn}

Of relevance to us is the fact that the principle $\SRP(\omega_1,S^{\omega_2}_\omega)$ is equivalent to the principle $\OSR{\omega_2}$ of Larson \cite{Larson:SeparatingSRP}. Larson showed that this principle follows from Martin's Maximum, but not from \SRP. Adding to this, in \cite[Thm.~4.4]{Fuchs-LambieHanson:SeparatingDSR}, it was shown that \SRP does not imply $\uDSR(1,S^\lambda_\omega)$, for $\lambda>\omega_2$, while \cite{Fuchs:DiagonalReflection} shows that for $\lambda>2^\omega$, \SCFA implies even the stronger principle $\DSR(\omega_1,S^\lambda_\omega)$. Thus, the strong reflection principles fail to capture these consequences of \MM/\SCFA, and our goal is to show that the \emph{diagonal} strong reflection principles do capture them; in fact, even $\wDRP_{\mathsf{IA}}$ is sufficient.

Note that the assumptions of the following theorem are satisfied if $\DSRP(\kappa)$ holds, or if $\kappa>2^\omega$ and $\infSC$-$\DSRP(\kappa)$ holds, by Corollary \ref{cor:DSRPimpliesDRPliftings}.

\begin{thm}
\label{thm:DRP(Lifting)=>DSR}
Let $\kappa>\omega_1$ be regular. Then
$\mathsf{w}\DRP_\IA(\mathcal{S}_\lifting(\kappa))\implies\DSR(\omega_1,S^\kappa_\omega)$.
\end{thm}

\begin{proof}
Let $\vS=\seq{S_{\alpha,i}}{\alpha<\theta,i<\omega_1}$ be a matrix of stationary subsets of $S^\kappa_\omega$. Let $\theta$ be a regular cardinal such that $\mathcal{S}_\lifting(\kappa)\sub H_\theta$. By $\mathsf{w}\DRP_\IA(\mathcal{S}_\lifting(\theta))$, let $W\prec\kla{H_\kappa,\in,\vS}$ satisfy clauses (1)(a) and (b) of Lemma \ref{lem:DSRPimpliesDRP}. 

Let $C=W\cap\kappa$, and let $\gamma=\sup(C)<\kappa$. Since $W$ is internally approachable, it can be written as $W=\bigcup_{i<\omega_1}W_i$, where $\vec{W}$ is a continuous elementary chain such that for all $i<\omega_1$, $\vec{W}\rest i\in W$. Thus, if we let $\theta_i=\sup(W_i\cap\theta)$, then $\bar{C}=\{\theta_i\st i<\omega_1\}$ is a closed unbounded subset of $C$, and since $\vec{\theta}$ is strictly increasing, the cofinality of $\gamma$ is $\omega_1$.

\begin{NumberedClaim}
{For every $\alpha\in C$, and for every $i<\omega_1$, $S_{\alpha,i}\cap\gamma$ is stationary in $\gamma$.}
\end{NumberedClaim}

To see this, fix $\alpha\in C$ and $i<\omega_1$. Note that $S_{\alpha,i}\in W$, and $\kappa$, being definable from $\vS$, is also in $W$. Hence, $\tilde{S}_{\alpha,i}=\lifting(S_{\alpha,i},[H_\kappa]^\omega)\in W\cap\mathcal{S}_{\lifting}(\kappa)$. 
	It follows that $\tilde{S}_{\alpha,i}\cap[W\cap H_\kappa]^\omega$ is stationary in $W\cap H_\kappa$. A standard argument shows that this, in turn, implies that $S_{\alpha,i}\cap\gamma$ is stationary in $\gamma$. In detail, let $D\sub\gamma$ be club. Let $E=D\cap\bar{C}$. Since $\cf(\gamma)>\omega$, $E$ is club. Let $f:C\To C$ be defined by $f(\xi)=\min(E\ohne(\xi+1))$. Since $\tilde{S}_{\alpha,i}\cap[W\cap H_\kappa]^\omega$ is stationary, we can pick a set $x\in\tilde{S}_{\alpha,i}\cap[W\cap H_\kappa]^\omega$ closed under $f$. Let $\delta=\sup(x\cap\kappa)$. Since $x\in\tilde{S}_{\alpha,i}$, it follows that $\delta\in S_{\alpha,i}$. By definition of $f$, $\delta$ is clearly a limit point of $E$, hence also a limit point of $D$. So $\delta\in(S_{\alpha,i}\cap\gamma)\cap D$.

Thus, the club set $\bar{C}$ witnesses this instance of $\DSR(\omega_1,S^\theta_\omega)$.
\end{proof}

%

\subsection{Consequences beyond \DRP: exact reflection}
\label{subsec:Unfiltered}

It was pointed out in the beginning of this section that principles of exact reflection postulate the existence of points at which some stationary sets reflect, but others don't. The fact that the weak diagonal reflection principle is monotonic in its argument (see Remark \ref{rem:Monotonicity}) is an indication that it does not capture these kinds of exact reflection. The proof of the following observation shows that sometimes, it is useful to have $\DSRP(\mathcal{S},\vT)$ for a very small collection $\mathcal{S}$ of stationary sets indeed.

\begin{obs}
\label{obs:Gamma-DSRPimpliesGamma-SRP}
Let $\Gamma=\SSP$ or $\Gamma=\infSC$. Then $\Gamma$-\DSRP implies $\Gamma$-\SRP.
\end{obs}

\begin{proof}
Let $\kappa\ge\omega_2$, and let $S\sub[H_\kappa]^\omega$ be stationary in $H_\kappa$ and $\Gamma$-projective stationary. Let $\vT=\seq{T_i}{i<\omega_1}$ be a partition of $\omega_1$ into stationary sets. Let $\mathcal{S}=\{S\}$. Then $\kla{\mathcal{S},\vT}$ is $\Gamma$-projective stationary (if $\Gamma=\SSP$, then this is by Definition \ref{def:ProjectiveStationaryOnVecT} and Theorem \ref{thm:SSP-AdequacyIsPSonVecT}, and in case $\Gamma=\infSC$, it follows from Definition \ref{def:SpreadOutOnVecT} and Theorem \ref{thm:CharacterizationOfInfSCAdquacy}). Thus, by $\Gamma$-\DSRP, $\DSRP(\mathcal{S},\vT)$ holds, but if $\kla{\vec{Q},\vS}$ witnesses this, then $\vQ$ witnesses the required instance of $\Gamma$-\SRP.
\end{proof}

It will follow from results in Section \ref{sec:DRPlimitations} that an internally approachable form of the original principle \DRP, which strengthens the principles of the form $\wDRPIA(\kappa)$, 
does not imply \SRP. 

Coming up is a typical example of a consequence of a \DSRP type assumption. To make exact reflection meaningful, we have to add in a constraint, but modulo this constraint, we get maximal reflection.

\begin{lem}
Let $\kappa>\omega_1$ be a regular cardinal, and let $E\sub S^\kappa_\omega$ be stationary in $\kappa$. Let
\[\mathcal{S}=\{\lifting(A,[H_\kappa]^\omega)\st A\sub E\ \text{is stationary in $\kappa$}\}.\]
Let $\theta$ be a sufficiently large cardinal so that $\mathcal{S}\sub H_\theta$. Then $\eDSRP(\mathcal{S})$ implies that for stationarily many $W\in[H_\theta]^{\omega_1}$, we have that $\omega_1\sub W$ and $\rho=\sup(W\cap\kappa)$ is an exact simultaneous reflection point of $\{A\in W\st A\sub E\ \text{and $A$ is stationary in $\kappa$}\}$.
\end{lem}

\noindent\emph{Note:} Again, the assumptions of this lemma hold if $\SSP$-$\DSRP(\kappa)$ holds, or if $\kappa>2^\omega$ and \infSC-$\DSRP(\kappa)$ holds.

\begin{proof}
Let $W\in[H_\theta]^{\omega_1}$ be such that $W\prec\kla{H_\theta,\mathcal{S}}$, $\omega_1\sub W$, and such that there is an exact diagonal chain $\vQ$ through $\mathcal{S}$ up to $W$. By $\eDSRP(\mathcal{S})$, there are stationarily many such. It then follows in a straightforward way that the set $\{\sup(Q_i\cap\kappa)\st i<\omega_1\}$ is a club subset of $\bigcup\{A\in W\st A\sub E\ \text{is stationary}\}$, and since for every $A\sub E$ stationary in $\kappa$ that exists in $W$, the set $S=\lifting(A,[H_\kappa]^\omega)\in\mathcal{S}\cap W$, we have that for stationarily many $i<\omega_1$, $Q_i\in S$, which means that $\sup(Q_i\cap\kappa)\in A$, it follows that $\rho$ is a reflection point of $A$. 	
\end{proof}

The previous lemma is of course most interesting if $S^\kappa_\omega\ohne E$ is also stationary in $\kappa$.
Let us now strengthen the diagonal reflection principles for sequences of stationary sets of ordinals, as given in Definition \ref{defn:DiagonalReflection}, so as to arrive at their exact versions, focusing on the variants most relevant for our purposes.

\begin{defn}
\label{defn:ExactDiagonalReflection}
Let $\kappa$ be a regular cardinal, and let $S\subseteq\kappa$ be stationary. An $(\omega_1,S)$-sequence is a sequence $\seq{S_{\alpha,i}}{\alpha<\kappa, i<\omega_1}$ of subsets of $S$ stationary in $\kappa$. Given such a sequence $\vS$, an ordinal $\rho<\kappa$ of uncountable cofinality is an \emph{exact diagonal reflection point} of $\vS$ if there is a set $R\sub\rho$ such that
\begin{enumerate}[label=(\arabic*)]
	\item
	\label{item:RhasCardOmega1}
	$R$ has cardinality $\omega_1$,
	\item
	\label{item:RcontainsAclub}
	$R$ contains a club in $\rho$,
	\item
	\label{item:RhoIsAnExactSimultaneousReflectionPoint}
	$\rho$ is an exact simultaneous reflection point of $\{S_{\alpha,i}\st\alpha\in R, i<\omega_1\}$.
\end{enumerate}
The \emph{exact diagonal reflection principle} $\eDSR(\omega_1,S)$ says that every $(\omega_1,S)$-sequence has an exact diagonal reflection point.
\end{defn}

Recall that even the simple (non-exact) diagonal reflection principle $\DSR(\omega_1,S^{\omega_2}_{\omega})$, does not follow from \SRP; see the discussion after Definition \ref{defn:DiagonalReflection}. The following theorem shows that \DSRP implies the exact version.

\begin{thm}
\label{thm:DSRP=>eDSR}
Let $\kappa>\omega_1$ be regular, and let $\vS=\seq{S_{\alpha,i}}{\alpha<\kappa, i<\omega_1}$ be an $(\omega_1,S^\kappa_\omega)$-sequence. Let $\mathcal{S}=\{\lifting(S_{\alpha,i},[H_\kappa]^\omega)\st\alpha<\kappa,\ i<\omega_1\}$. Then $\eDSRP(\mathcal{S})$ implies the existence of an exact diagonal reflection point for $\vS$.
\end{thm}

\noindent\emph{Note:} By Lemma \ref{thm:Gamma-DSRPimpliesExactDSRP}, the assumption of this theorem holds if $\kappa>2^\omega$ and $\infSC$-$\DSRP(\kappa)$ holds, or if $\DSRP(\kappa)$ holds. That is, we have that ``$\infSC$-$\DSRP(\kappa) + \kappa>2^\omega$ is regular'' implies $\eDSR(\omega_1,S^\kappa_\omega)$, as does ``$\DSRP(\kappa) + \kappa>\omega_1$ is regular''.

\begin{proof}
Let $\theta$ be a cardinal such that $\mathcal{S}\sub H_\theta$. Let $\omega_1\sub W\prec\kla{H_\theta,\in,\vec{S}}$ have size $\omega_1$ and let $\vQ$ be an exact diagonal chain through $\mathcal{S}$ up to $W$. Let $R=W\cap\kappa$ and $\rho=\sup(R)$. We claim that $\rho$ is an exact diagonal reflection point of $\vS$, as witnessed by $R$. Let, for $j<\omega_1$, $\rho_j=\sup(Q_j\cap\kappa)$. Then $\rho=\sup_{j<\omega_1}\rho_j$ and $C=\{\rho_j\st j<\omega_1\}$ is club in $\rho$, $\rho$ has cofinality $\omega_1$, and $C\sub R$. This verifies conditions \ref{item:RhasCardOmega1} and \ref{item:RcontainsAclub}.


To see that $\rho$ is a reflection point of $S_{\alpha,i}$, for every $\alpha\in R$ and every $i<\omega_1$, fix such $\alpha$ and $i$. Let $T_{\alpha,i}=\{j<\omega_1\st Q_j\in\lifting(S_{\alpha,i},[H_\kappa]^\omega)\}$.
Since $\alpha, i\in W$, it follows that $\lifting(S_{\alpha,i},[H_\kappa]^\omega)\in W\cap\mathcal{S}$, and so, $T_{\alpha,i}$ is stationary in $\omega_1$.
But whenever $j\in T_{\alpha,i}$, then $\rho_j\in S_{\alpha,i}$. So since $T_{\alpha,i}$ is stationary in $\omega_1$ and the map $j\mapsto\rho_j$ is continuous and strictly increasing, it follows that $\{\rho_j\st j\in T_{\alpha,i}\}$ is stationary in $\rho$. Since $\{\rho_j\st j\in T_{\alpha,i}\}\sub S_{\alpha,i}$, it follows that $\rho$ is a reflection point of $S_{\alpha,i}$. Thus, $\rho$ is a simultaneous reflection point of $\{S_{\alpha,i}\st\alpha\in R, i<\omega_1\}$.

Finally, since $\vQ$ is exact, we have that for every $j<\omega_1$, $Q_j\in S$, for some $S\in\mathcal{S}\cap W$, and hence, $\rho_j\in S_{\alpha,i}$, for some $\alpha\in W\cap\kappa$ and some $i<\omega_1$. This is because if $S\in\mathcal{S}\cap W$, then since $W\prec\kla{H_\theta,\in,\vS}$, there is a least pair $\kla{\alpha,i}$ such that $S=S_{\alpha,i}$, which must be in $W$. Thus, $C\sub\bigcup_{\alpha\in R, i<\omega_1}S_{\alpha,i}$, verifying the ``exactness'' part of condition \ref{item:RhoIsAnExactSimultaneousReflectionPoint}.
\end{proof}

As a last example, let us state an exact diagonal mutual reflection principle with a constraint. The formulation is a little tedious, but the principle is quite natural.

\begin{defn}
\label{def:StationarySetsAndTheirRestrictions}
Let $\kappa$ be an ordinal of uncountable cofinality. We write $\Stationary_\kappa$ for the collection of all stationary subsets of $\kappa$. Given a set $E\sub\kappa$, we write $\Stationary_\kappa\rest E$ for the collection of all subsets of $E$ that are stationary in $\kappa$.	
\end{defn}

\begin{thm}
Let $K\neq\leer$ be a set of regular cardinals greater than $\omega_1$, with supremum $\tkappa$. Let $\seq{E_\kappa}{\kappa\in K}$ be a sequence of sets such that for each $\kappa\in K$, $E_\kappa\sub S^\kappa_\omega$ is stationary in $\kappa$. Now, for every $\vA\in\prod_{\kappa\in K}\mathsf{S}_\kappa\rest E_\kappa$, let
\[S_\vA=\{X\in[H_\tkappa]^\omega\st\forall\kappa\in X\cap K\quad\sup(X\cap\kappa)\in A_\kappa\}\]
and let
\[\mathcal{S}=\{S_{\vA}\st\vA\in\prod_{\kappa\in K}\Stationary_\kappa\rest E_\kappa\}.\]
Assume $\eDSRP(\mathcal{S})$ holds and let $\theta$ be a cardinal sufficiently large so that $\mathcal{S}\sub H_\theta$. Then there are stationarily many $W\in[H_\theta]^{\omega_1}$ such that:
\begin{enumerate}[label=\em(\arabic*)\em]
\item $\omega_1\sub W$.
\item For every $\kappa\in K\cap W$, $\rho_\kappa=\sup(W\cap\kappa)$ is an exact simultaneous reflection point for $(\Stationary_\kappa\rest E_\kappa)\cap W$.
\item There is a matrix $\seq{\rho_{\kappa,i}}{\kappa\in K\cap W, i<\omega_1}$ such that:
    \begin{enumerate}[label=\em(\alph*)\em]
    	\item for every $\kappa\in K\cap W$, there is a $\delta_\kappa<\omega_1$ such that the function $\delta_\kappa<i\mapsto\rho_{\kappa,i}$ is strictly increasing, continuous and cofinal in $\rho_\kappa$.
    	\item For every $\kappa\in K\cap W$ and every sufficiently large $i<\omega_1$,
		\[\rho_{\kappa,i}\in W\cap E_\kappa.\]
    	\item For every $\vA\in W\cap\prod_{\kappa\in K}\Stationary_\kappa\rest E_\kappa$, there is a stationary subset $T_{\vA}\sub\omega_1$ such that for every $\kappa\in K\cap W$ and every sufficiently large $i\in T_{\vA}$, $\rho_{\kappa,i}\in A_\kappa$.
    \end{enumerate}
\end{enumerate}
\end{thm}

\begin{remark}
It was shown in \cite[Cor.~3.26]{Fuchs:CanonicalFragmentsOfSRP} that $S_{\vA}$ is projective stationary, and so, the assumptions of the theorem follow from $\DSRP$. By \cite[Cor.~3.32]{Fuchs:CanonicalFragmentsOfSRP}, $S_\vA$ is even spread out if $\infSC$-$\SRP + \CH$ holds, so the assumptions also follow from $\infSC$-$\DSRP + \CH$. 	
\end{remark}

\begin{proof}
By assumption, there are stationarily in $H_\theta$ many $W\in[H_\theta]^{\omega_1}$ with $\omega_1\sub W$ such that there is an exact diagonal chain $\seq{Q_i}{i<\omega_1}$ through $\mathcal{S}$ up to $W$. Define, for $\kappa\in K\cap W$ and $i<\omega_1$, $\rho_{\kappa,i}=\sup(Q_i\cap\kappa)$. It is routine to check that all the conditions are satisfied.
\end{proof}

\section{Limitations}
\label{sec:DRPlimitations}

In this section, we will present some negative results, separating some of the principles under investigation. The first of these employs methods of Miyamoto.

\begin{thm}
\label{thm:SeparatingDSRPfromMM}
Assuming the consistency of a supercompact cardinal, $\DSRP$ does not imply $\MM$; it is consistent with the existence of a Souslin tree.
\end{thm}

\begin{proof}
Miyamoto \cite[Def.~5.4]{Miyamoto:IteratingSemiproperPreorders} introduced the forcing axiom $\MM(Souslin)$, which is the forcing axiom for the class of all stationary set preserving forcing notions that also preserve every $\omega_1$-Souslin tree, and he showed \cite[Cor.~5.8]{Miyamoto:IteratingSemiproperPreorders} that assuming the consistency of a supercompact cardinal, $\MM(Souslin)+$``there is a Souslin tree'' is consistent. He also showed that $\MM(Souslin)$ implies \SRP. All we have to do is observe that $\MM(Souslin)$ also implies $\DSRP$. For this, it clearly suffices to show, given a pair $\kla{\mathcal{S},\vT}$ that is \SSP-projective stationary, that $\P=\P^\DSRP_{\mathcal{S},\vT}$ preserves Souslin trees. So let $U$ be a Souslin tree, let $p\in\P$, and suppose that $\dot{A}$ is a $\P$-name such that $p$ forces that $\dot{A}$ is a maximal antichain in $T$. We have to find an extension $q$ of $p$ that forces $\dot{A}$ to be countable. We may assume that $S=S^p_0$ is defined. So $S$ is projective stationary on $T_0$. Let $\theta$ be a sufficiently large regular cardinal, and let $X\prec\kla{H_\theta,\in,<^*}$ with $\P,\mathcal{S},p,\vT,U\in X$, and such that $\delta=X\cap\omega_1\in T_0$ and $X\cap H_\kappa\in S$.

Let $\{t_n\st n<\omega\}$ enumerate the $\delta$-th level of $U$, and define
\begin{ea*}
D_n&=&\{r\in \P\cap M\st \text{either $r$ is incompatible with $p$, or there is a $t\in T\cap M$}\\
&& \qquad\qquad\qquad\text{such that $t<_U t_n$ and $r\forces\check{t}\in\dot{A}$}\}.
\end{ea*}%
The point is that $D_n$ is dense in $\P\cap M$, as is shown by the argument of \cite[claim on p.~1464]{Miyamoto:IteratingSemiproperPreorders}: let $a\in\P\cap M$ be given. If $a$ is incompatible with $p$, then $a\in D_n$ and we are done. Otherwise, by strengthening $a$, we may assume that $a\le_\P p$. Let $D=\{t\in U\st\exists r\le_\P a\ r\forces\check{t}\in\dot{A}\}$. Since $a$ forces $\dot{A}$ to be a \emph{maximal} antichain in $U$, it is easy to see that $D$ is predense in $U$, that is, every element of $U$ is comparable with some member of $D$. Since $U$ is a Souslin tree, the set $b_n=\{t\in U\st t<_U t_n\}\cap M$ is $U$-generic over $M$, and hence, it intersects $D$ in $M$, as $D\in M$. So let $t\in b_n\cap D\cap M$. Let $r\le a$ witness that $t\in D$. Then $r\in D_n$, as witnessed by $t$.

Note that $D_n$ is not in $M$. But we may construct an $M$-generic $G$ by forming a decreasing chain $\seq{p_n}{n<\omega}$ of conditions in $\P$ such that, letting $\seq{E_n}{n<\omega}$ enumerate all dense open subsets of $\P$ in $M$, $p_n\in D_n\cap E_n$ (note that $D_n$ is open as well), and such that $p_0\le p$. Let $G$ be the filter generated by $\vp$. Then, letting $\vQ=\bigcup_{n<\omega}\vQ^{p_n}$ and $\vS=\bigcup_{n<\omega}\vS^{p_n}$, we have that $\delta$ is the length of $\vQ$, which is the same as the length of $\vS$, and $M\cap H_\kappa=\bigcup_{i<\delta}Q_i\in S_0$. Since $\delta\in T_0$, we can define a condition $q$ by setting
\[q=\kla{\vQ\verl M,\vS}.\]
Clearly, $q$ forces that $\dot{A}$ is contained in $U\rest\delta$, the restriction of $U$ to levels below $\delta$.
\end{proof}

It is now natural to ask for a similar separation between $\Gamma$-\DSRP and $\FA(\Gamma)$, where $\Gamma$ is the class of all subcomplete or all $\infty$-subcomplete forcing notions.
It was observed in \cite{Fuchs:CanonicalFragmentsOfSRP} that, assuming the consistency of \MM, $\infSC$-\SRP does not imply \SCFA, since under the assumption of \MM, a model of \ZFC can be constructed in which $\SRP + \neg\uDSR(1,S^{\omega_3}_\omega)$ holds. So this model satisfies \infSC-\SRP, but not \SCFA, or else it would have to satisfy $\DSR(\omega_1,S^{\omega_3}_\omega)$. For the same reason, though, $\infSC$-\DSRP fails in this model as well, so this method does not separate $\infSC$-\DSRP from $\SCFA$. Theorem \ref{thm:SeparatingDSRPfromMM} does not achieve this separation either, because \SCFA is consistent with the existence of Souslin trees. Further, it was argued in \cite{Fuchs:CanonicalFragmentsOfSRP} that the assumption of \CH should be added to $\infSC$-\SRP, since for regular $\kappa\in(\omega_1,2^\omega)$, $\infSC$-$\SRP(\kappa)$ holds trivially. Since the models achieving the separations up to now satisfied \SRP, \CH fails in them, and so, they don't achieve a separation of this kind.

The following theorem does achieve a certain separation at the level $\omega_2$ in the presence of \CH. This result was alluded to at the end of the article \cite{Fuchs:CanonicalFragmentsOfSRP}, but not made precise.
For this result, it is important that we work with subcompleteness, not $\infty$-subcompleteness. Since we will be using results of \cite{Fuchs:CanonicalFragmentsOfSRP} as a black box, the reasons for this will remain obscure here; let us just say that the problem is the iteration theorem \cite[Thm.~4.17]{Fuchs:CanonicalFragmentsOfSRP}. The exact relationship between subcompleteness and $\infty$-subcompleteness is not well understood, but subcompleteness is a potentially more restrictive requirement than $\infty$-subcompleteness, so that the principle \infSC-\DSRP could be stronger than \SC-\DSRP. But all the consequences of $\infSC$-\DSRP presented in Section \ref{subsec:DSRPconsequences} also follow from \SC-\DSRP, and the subcomplete fragment of \DSRP can be characterized by replacing ``spread out'' with ``fully spread out'' everywhere (see \cite[Def.~2.33]{Fuchs:CanonicalFragmentsOfSRP}.) The following definition summarizes the concepts needed for the statement of the result.

\begin{defn}
For a forcing class $\Gamma$ and a cardinal $\kappa$, $\BFA(\Gamma,{\le}\kappa)$, the \emph{${\le}\kappa$-bounded forcing axiom for $\Gamma$,} says that if $\P\in\Gamma$ and $\B$ is the complete Boolean algebra of $\P$, and if $\mathcal{A}$ is a collection of at most $\omega_1$ many maximal antichains in $\B$, each of which has cardinality at most $\kappa$, then there is a filter in $\B$ that meets each antichain in $\mathcal{A}$. We write $\BSCFA({\le}\kappa)$ in case $\Gamma$ is the class of all subcomplete forcing notions.

For a regular cardinal $\kappa\ge\omega_2$ and an uncountable cardinal $\lambda$, the principle \SC-$\DSRP(\kappa,\lambda)$ asserts that whenever $\mathcal{S}$ is a nonempty collection of subsets of $[H_\kappa]^\omega$ that are stationary in $H_\kappa$, such that $\mathcal{S}$ has size at most $\lambda$, and $\vT$ is a sequence of pairwise disjoint stationary subsets of $\omega_1$, and $\kla{\mathcal{S},\vT}$ is \SC-projective stationary, then $\DSRP(\mathcal{S},\vT)$ holds.
\end{defn}

Before moving to the separation result, let us make an observation related to the two cardinal version of \DSRP introduced in the previous definition.

\begin{obs}
Let $\Gamma$ be $\SSP$, $\infSC$ or \SC. Let $\mathcal{S}$ be a collection of up to $\omega_1$ many sets $\Gamma$-projective stationary in $H_\kappa$, for some regular $\kappa\ge\omega_2$. Then $\Gamma$-$\SRP(\kappa)$ implies $\eDSRP(\mathcal{S})$.
\end{obs}

\begin{proof}
Let $\vT$ be a maximal partition of $\omega_1$ into stationary sets, and let $\seq{S_i}{i<\omega_1}$ enumerate $\mathcal{S}$. Let
\[S=\{x\in[H_\kappa]^\omega\st\forall i<\omega_1\ x\cap\omega_1\in T_i \longrightarrow x\in S_i\}.\]
\noindent\emph{Claim:} $S$ is $\Gamma$-projective stationary.

\noindent\emph{Case 1:} $\Gamma=\SSP$.

Then $\Gamma$-projective stationarity is just projective stationarity. So let $A\sub\omega_1$ be stationary. By maximality of $\vT$, let $i_0<\omega_1$ be such that $A\cap T_{i_0}$ is stationary. Since $S_{i_0}$ is projective stationary,
\[\{x\in S_{i_0}\st x\cap\omega_1\in A\cap T_{i_0}\}\]
is stationary. But this set is contained in $\{x\in S\st x\cap\omega_1\in A\}$, making the latter set stationary, and hence $S$ is projective stationary.

\noindent\emph{Case 2:} $\Gamma=\infSC$.

Then $\Gamma$-projective stationarity is being spread out. So let $\theta$ be a sufficiently large cardinal, $H_\theta\sub N=L_\tau^A\models\ZFCm$, $N|X\prec N$, $X$ countable and full, $\vec{S},\vT,S,a\in X$. Since $\vT$ is maximal, $Z=\omega_1\ohne\bigtriangledown_{i<\omega_1}T_i$ is not stationary. Since $Z\in X$, it follows that $\delta=X\cap\omega_1\notin Z$. So $\delta\in\bigtriangledown_{i<\omega_1}T_i$. Let $\delta\in T_{i_0}$. Then $i_0<\delta$. So $S_{i_0}\in X$. Since $S_{i_0}$ is spread out and $S_{i_0}\in X$, let $\pi:N|X\To N|Y\prec N$ be an isomorphism fixing $\vec{S},\vT,S,a$, such that $Y\cap H_\kappa\in S_{i_0}$. Since $Y\cap\omega_1=X\cap\omega_1=\delta\in T_{i_0}$, it follows that $Y\cap H_\kappa\in S$, verifying that $S$ is spread out.

\noindent\emph{Case 3:} $\Gamma=\SC$.

In this case, one has to work with fully spread out sets instead of spread out sets (see \cite[Def.~2.33]{Fuchs:CanonicalFragmentsOfSRP}). The argument of case 2 goes through.

This proves the claim. Thus, by $\Gamma$-$\SRP(\kappa)$, there is a continuous $\in$-chain of length $\omega_1$ through $S$, and this easily implies $\eDSRP(\mathcal{S})$.
\end{proof}

\begin{thm}
Let $\Gamma$ be the class of all subcomplete, uncountable cofinality preserving forcing notions. If \ZFC is consistent with $\BFA(\Gamma,{\le}\omega_2)$, then \ZFC is consistent with the conjunction of the following statements:
\begin{enumerate}[label=(\arabic*)]
  \item
  \label{item:CH}
  \CH,
  \item
  \label{item:BFAGammaOmega2}
  $\BFA(\Gamma,{\le}\omega_2)$,
  \item
  \label{item:NotBSCFAOmega2}
  $\neg\BSCFA(\omega_2)$,
  \item
  \label{item:DSRPOmega2Omega2}
  \SC-$\DSRP(\omega_2,\omega_2)$.
\end{enumerate}
\end{thm}

\noindent\emph{Note:} $\SC$-$\DSRP(\omega_2,\omega_2)+\CH$ has interesting consequences that go beyond \SC-$\SRP(\omega_2)$. For example, it implies $\eDSR(\omega_2,S^{\omega_2}_\omega)$, see Theorem \ref{thm:DSRP=>eDSR} - the collection $\mathcal{S}$ used in the proof of this theorem has size $\kappa=\omega_2$ in our situation.

\begin{proof}
It was shown in \cite[Thm.~4.25]{Fuchs:CanonicalFragmentsOfSRP} that under the assumptions of the theorem, there is a model in which \ZFC holds, together with \ref{item:CH}-\ref{item:NotBSCFAOmega2}. So it suffices to show that \ref{item:CH} + \ref{item:BFAGammaOmega2} implies \ref{item:DSRPOmega2Omega2}.

To see this, let $\kla{\mathcal{S},\vT}$ be $\SC$-projective stationary, where $\mathcal{S}$ consists of subsets of $[H_\kappa]^\omega$ stationary in $H_\kappa$ and has cardinality at most $\omega_2$. Let $\theta$ be large enough that $\mathcal{S}\sub H_\theta$, and let \[\calM=\kla{H_\theta,\in,\mathcal{S},F,<^*,\vT,0,1,\ldots,\xi,\ldots}\]
be a model of a language of size $\omega_1$ with some extra predicate $F$, a well-order $<^*$, constant symbols $\dot{\xi}$ for every countable ordinal $\xi$, and with a constant symbol for $T_i$, for every $i<\omega_1$. We have to find an $M\prec\calM$ of size $\omega_1$, with $\omega_1\sub M$, such that there is a diagonal chain through $\mathcal{S}$ up to $M$ wrt.~$\vT$. Let $\bar{\calM}\prec\calM$ be the transitive collapse of the hull of $H_{\omega_2}\cup\mathcal{S}$ in $\calM$. So $\bar{\calM}$ has cardinality $\omega_2$, since $2^{\omega_1}=\omega_2$ - it was shown in \cite[Lemma 4.24.(2)]{Fuchs:CanonicalFragmentsOfSRP} that $\BFA(\Gamma,{\le}\omega_2)$ implies \SC-$\SRP(\omega_2)$, and this, in turn, together with \CH, implies $2^{\omega_1}=\omega_2$, by \cite[Thm.~3.19]{Fuchs:CanonicalFragmentsOfSRP} and the following remarks, and \cite[Fact 3.15]{Fuchs:CanonicalFragmentsOfSRP}. Let $G$ be $\P=\P^{\DSRP}_{\mathcal{S},\vT}$-generic over $\V$. In $\V[G]$, let $\kla{\vQ,\vS}$ be the sequence added by $G$. Then $\bigcup_{i<\omega_1}Q_i=H_{\omega_2}$ and $\mathcal{S}=\{S_i\st i<\omega_1\}$. So in $\V[G]$, the following statement is true about $\bar{\calM}$: there are sequences $\vQ'$ and $\vS'$ of length $\omega_1^{\bar{\calM}}$ such that $\vQ'$ is a continuous $\in$-chain unioning up to $H_{\omega_2}^{\bar{\calM}}$, for every $i<\omega_1$, $S'_i\in\mathcal{S}^{\bar{\calM}}$, and if $j<\omega_1^{\bar{\calM}}$ is such that $i\in T^{\bar{\calM}}_j$, then $Q'_i\in S'_j$, and such that $\mathcal{S}^{\bar{\calM}}=\{S'_i\st i<\omega_1\}$. This is a $\Sigma_1$ statement about $\bar{\calM}$ forced to be true by $\P$, so since $\P\in\Gamma$ and $\BFA(\Gamma,{\le}\omega_2)$ holds, there are by \cite[Fact 4.21]{Fuchs:CanonicalFragmentsOfSRP} (see also \cite[Thm.~1.3]{ClaverieSchindler:AxiomStar}) a transitive model $\tilde{\calM}$ of the same language as $\bar{\calM}$ and an elementary embedding $j:\tilde{\calM}\prec\bar{\calM}$, so that the same $\Sigma_1$ statement is true about $\tilde{\calM}$ in $\V$. Note that $\omega_1\sub\tilde{\calM}$ and $j\rest\omega_1=\id$, since the language contains constant symbols for all the countable ordinals. If the witnessing sequences are $\vec{\tilde{Q}}$ and $\vec{\tilde{S}}$, then, letting $\pi:\bar{\calM}\To\calM$ be the inverse of the collapse, it follows that $\pi\rest H_{\omega_2}=\id$, and so, if we define $\vQ'$ by $Q'_i=j(\tQ_i)$ and $\vS'$ by $S'_i=j(\tilde{S}_i)$, then $\kla{\vQ',\vS'}$ is a diagonal chain through $\mathcal{S}$ up to $M=\ran(\pi\circ j)$ with respect to $\vT$, where $\omega_1\sub M$ and $M\prec\calM$.
\end{proof}

The last separation result concerns the diagonal reflection principle of \cite{Cox:DRP} and its relationship to cardinal arithmetic. This principle is stronger than the principles of the form $\wDRPIA(\kappa)$ we have considered. 

\begin{defn}
Let $\theta$ be an uncountable regular cardinal. The principle $\DRP(\theta,\IA)$ states that there are stationarily many $M\in[H_{(\theta^\omega)^+}]^{\omega_1}$ such that $M\cap H_\theta\in\IA$ and for every stationary subset $R\in M$ of $[\theta]^\omega$, $R\cap[M\cap\theta]^\omega$ is stationary in $M\cap\theta$.

The principle $\DRP(\IA)$ states that $\DRP(\theta,\IA)$ holds for all regular $\theta\ge\omega_2$.
\end{defn}

\begin{thm}
$\DRP(\IA)$ does not limit the size of $2^{\omega_1}$.
\end{thm}

\noindent\emph{Note:} This shows that this principle, if consistent, does not imply $\eRefl(\omega_1,S^{\omega_2}_\omega)$; see Fact \ref{fact:ExactReflectionAndCardinalArithmetic}. In particular, it does not imply \SRP.

\begin{proof}
We will show that the theory ``\CH plus the forcing axiom $\text{FA}^{+\omega_1}(\sigma \text{-closed})$'' is preserved after adding any number of Cohen subsets of $\omega_1$.  This will suffice, since
\begin{enumerate}
 \item $\text{FA}^{+\omega_1}(\sigma \text{-closed})$ implies $\DRP(\theta,\IA)$, for all regular $\theta>\omega_1$ (\cite[Theorem 4.1]{Cox:DRP}); and
 \item \CH is consistent with $\text{FA}^{+\omega_1}(\sigma \text{-closed})$ (\cite{FMS:MM1} shows that forcing with $\text{Col}(\omega_1,<\kappa)$ when $\kappa$ is supercompact produces a model satisfying this).
\end{enumerate}

So assume \CH plus $\text{FA}^{+\omega_1}(\sigma \text{-closed})$ both hold in $V$.  Pick any cardinal $\lambda$, and let $\mathbb{P}$ be the countable support product of $\lambda$-many copies of $\text{Add}(\omega_1)$. Since \CH holds, $\mathbb{P}$ has the $\omega_2$-cc, so in particular, $\mathbb{P}$ preserves all cardinals $\ge\omega_2$, and forces $2^{\omega_1} \ge \lambda$.  It remains to show that $\text{FA}^{+\omega_1}(\sigma \text{-closed})$ is preserved.

Let $p$ be any condition in $\mathbb{P}$, and $\dot{\mathbb{R}}$ be a $\mathbb{P}$-name for a $\sigma$-closed poset.  Then $\mathbb{P}*\dot{\mathbb{R}}$ is $\sigma$-closed.  Fix a regular $\theta$ such that $\mathbb{P}*\dot{\mathbb{R}} \in H_\theta$.  Since $\V$ models $\text{FA}^{+\omega_1}(\sigma \text{-closed})$, \cite[Theorem 4.5]{Cox:FAapproachabilitySSR} implies that in some generic extension $W$ of $V$, there is an elementary embedding $j: \V \prec N$
such that:
\begin{enumerate}
 \item $\text{crit}(j) = \omega_2^\V =:\kappa$;
 \item $j \restriction H^\V_\theta \in N$;
 \item $|H^\V_\theta|^N=\aleph_1$;
 \item $H^\V_\theta$ is an element of the (transitivized) wellfounded part of $N$; and
 \item There is some $G*H \in N$ that is generic over $V$ for $(\mathbb{P} \restriction p )*\dot{\mathbb{R}}$.
\end{enumerate}

Since $\mathbb{P}$ has the $\kappa$-cc in $\V$ and $\text{crit}(j)=\kappa$, the map $j \restriction \mathbb{P}: \mathbb{P} \to j(\mathbb{P})$ is a regular embedding; so if we let $G'$ be generic over $W$ for the poset $j(\mathbb{P})/j``G$, it follows that $G'$ extends $j``G$, and in $W[G']$ the map $j$ lifts to an elementary embedding
\[
\widetilde{j}: V[G]  \prec N[G'].
\]
Since $G*H$ was already in $N$ and was generic over $V$ for $\mathbb{P}*\dot{\mathbb{R}}$, then in particular $H \in N[G']$ and $H$ is generic over $V[G]$ for $\mathbb{R}=\dot{\mathbb{R}}_G$.  Also, since both $j \restriction H^V_\theta$ and $G'$ are elements of $N$, it follows that $\widetilde{j} \restriction H^V_\theta[G]$ is an element of $N[G']$.  Hence
by (the reverse direction of)  \cite[Theorem 4.5]{Cox:FAapproachabilitySSR}, the forcing axiom for $\mathbb{R}$ holds in $V[G]$.
\end{proof}

\section{Open questions}
\label{sec:Questions}

In Subsection \ref{subsec:Unfiltered}, we presented consequences of \DSRP that neither follow from \SRP nor from $\DRP_{\IA}$. However, we have not separated the conjunction of $\SRP$ and $\DRP_{\IA}$ from \DSRP, even though it seems unlikely that this conjunction implies \DSRP. It would be interesting to know how to do that.

\begin{question}
Does $\SRP+\mathsf{w}\DRP_{\IA}$ imply $\eDSR(\omega_1,S^{\omega_2}_\omega)$, or even $\DSRP$?
\end{question}

Regarding the separation of \SC-$\DSRP(\omega_2,\omega_2)$ from $\BSCFA(\omega_2)$, it would be interesting to know if this can be improved.

\begin{question}
Can one show that \SC-$\DSRP(\omega_2)+\CH$ does not imply $\BSCFA(\omega_2)$? That \SC-$\DSRP+\CH$ does not imply \SCFA? How about the $\infSC$-versions of these separations?
\end{question}

\bibliographystyle{abbrv}
\bibliography{references.bib}
\end{document}